\documentclass[12pt]{amsart}

\usepackage{comment}

\newcommand{\BP}{\mathcal{W}^k }
\newcommand{\SBP}{\mathcal{W}_ k }

\newcommand{\da}[1]{{\color{blue}{D: #1}}}

\usepackage[a4paper, total={6in, 8in}]{geometry}

 \usepackage{graphicx}
\usepackage{bbm}
\usepackage[title]{appendix}
\usepackage{color}
\usepackage{tikz-cd}

 \usepackage{graphicx}
\usepackage{tikz-cd}

\newcommand{\bea}{\begin{eqnarray}}
\newcommand{\eea}{\end{eqnarray}}

\allowdisplaybreaks

\newcommand{\Z}{\Bbb Z}

\newcommand{\C}{\Bbb C}

\newcommand{\g}{  \mathfrak g}

\newcommand{\Ker}{\operatorname{Ker}}

\newcommand{\NO}[1]{:\!#1\!:}

\newtheorem{theorem}{Theorem}[section]
\newtheorem{lemma}[theorem]{Lemma}
\newtheorem{conjecture}[theorem]{Conjecture}
\theoremstyle{definition}

\theoremstyle{remark}
\newtheorem{remark}[theorem]{Remark}

\theoremstyle{proposition}
\newtheorem{proposition}[theorem]{Proposition}
\theoremstyle{corollary}
\newtheorem{corollary}[theorem]{Corollary}
\numberwithin{equation}{section}

\newcommand{\abs}[1]{\lvert#1\rvert}

\begin{document}

\title[]{ Relaxed and logarithmic modules of $\widehat{\mathfrak{sl}_3}$  }
\author[]{Dra\v zen  Adamovi\' c}
\author[]{Thomas Creutzig}
\author[]{ Naoki Genra}
 
\maketitle

\begin{abstract}
In \cite{A-2019}, the affine vertex algebra $L_k(\mathfrak{sl}_2)$ is realized as a subalgebra of the vertex algebra $Vir_c \otimes \Pi(0)$, where $Vir_c $
is a  simple Virasoro vertex algebra and $\Pi(0)$ is a half-lattice vertex algebra. Moreover, all $L_k(\mathfrak{sl}_2)$--modules (including,  modules in the category $KL_k$, relaxed highest weight modules and logarithmic modules) are realized as $Vir_c \otimes \Pi(0)$--modules. 

A natural question is the generalization of this construction in higher rank. In the current  paper, we study the case 
$\mathfrak{g}= \mathfrak{sl}_3$ and present realization of the VOA $L_k(\g)$ for $k \notin {\Z}_{\ge 0}$  as a vertex subalgebra of $\SBP \otimes \mathcal S \otimes \Pi(0)$, where  $\SBP$ is a simple Breshadsky Polykov vertex algebra and $\mathcal S$ is the $\beta \gamma$ vertex algebra. 

We use this realization to study ordinary modules, relaxed highest weight modules and logarithmic modules. We prove the irreducibility  of all our relaxed highest weight modules having  finite-dimensional weight spaces (whose top components are  Gelfand-Tsetlin modules). The irreducibility of relaxed highest weight modules with infinite-dimensional weight spaces is proved up to a conjecture on the irreducibility of certain $\mathfrak g$--modules which are not Gelfand-Tsetlin  modules.

The next problem that we consider is the realization of logarithmic modules. We first analyse the free-field realization of $\SBP$ from \cite{AKR-2020}  and obtained a realization of logarithmic modules for $\SBP$ of nilpotent rank two at most admissible levels. Beyond admissible levels, we get realization of logarithmic modules up to a existence of certain $\mathcal W_k(\mathfrak{sl}_3, f_{pr})$--modules. Using logarithmic modules for the $\beta \gamma$ VOA, we are able to construct logarithmic $L_k(\g)$--modules of rank three.
\end{abstract}

\section{Introduction}

Vertex algebras are rich algebraic structures with interesting applications in mathematics and physics. Given a vertex algebra, usually the major goal is a thorough understanding of its representation theory. While strongly rational vertex operator have semisimple categories of modules, that are in fact modular tensor categories \cite{Huang}, the generic vertex algebra is neither rational nor lisse. Basic vertex operator algebras arise via standard constructions from  affine vertex algebras, those associated to affine Lie algebras. One of these standard constructions is quantum Hamiltonian reduction, which associates to the affine vertex algebra $V^k(\mathfrak g)$ at level $k$ and the nilpotent element $f$ in the simple Lie algebra $\mathfrak g$  a $\mathcal W$-algebra $\mathcal W^k(\mathfrak g, f)$ \cite{KRW-03}. Usually one is then interested in the simple quotient  $\mathcal W_k(\mathfrak g, f)$ of $\mathcal W^k(\mathfrak g, f)$.
Representation categories of the $\mathcal W$-algebras are somehow nicer than their parent affine vertex algebras. e.g. simple principal $\mathcal W$-algebras at non-degenerate admissible levels are rational and lisse \cite{Ara-rat, Ara-C2}.

We are in general interested in the representation theory of the simple affine vertex algebra $L_k({\mathfrak g})$ at the admissible level $k$. 
This topic has received considerable attention in the recent years \cite{A-2005, A-2019, AdM-1995, ACK, ACR, AFR, CKLR, CR1, CR2, FL, KR1, KR2, KRW-21}.
The only (and only recently) fairly well understood case is the case of $\mathfrak g = \mathfrak{sl}_2$ and so our aim is to lift constructions to $\mathfrak{sl}_3$. 

Let us start by reviewing the $\mathfrak{sl}_2$ case in order to stress the importance of free field realizations:

One is interested in the category of smooth weight modules at level $k$ of the affine Lie algebra $\widehat{\mathfrak{sl}}_2$ that are also modules for $L_k(\mathfrak{sl}_2)$. By weight we mean that the Cartan subalgebra of the horizontal subalgebra $\mathfrak{sl}_2$ acts semisimply. We assume that $k$ is admissible, that is $k+2$ is a positive rational number and excluding $1/n, n \in \mathbb Z$.

The first step is to start with the interesting subcategory of modules whose conformal weight is lower bounded. Simple modules in this category are in one-to-one correspondence with simple modules of Zhu's algebra and they were classified 26 years ago \cite{AdM-1995}. 
For general simple Lie algebras see the much more recent \cite{KR1}. 

The second step is to twist modules by spectral flow, automorphisms induced from the extended affine Weyl group. 
In the case of $\mathfrak{sl}_2$ it  is  an easy computation to verify that every simple smooth module is a spectral flow twist of a lower bounded module and the same statement is true for general simple Lie algebras \cite{ACK} using a result of \cite{FT-2001}. 

The much harder question is a full classification of indecomposable modules. This will only appear in the forthcoming work \cite{ACK} and the most important ingredient is the construction of indecomposable modules, in particular logarithmic modules. A logarithmic module has a nilpotent action of the Virasoro zero mode. The single known construction for this is using free field realizatons \cite{A-2019}. It turns out that the full classification of indecomposable modules follows from a classification of simple modules, knowledge of existence of sufficiently many indecomposable modules, in particular logarithmic modules, and a few general criteria on extensions. 

Other benefits of the free field realization of  \cite{A-2019} are character formula (that were also obtained by different means \cite{KR2}) and existence of non-trivial intertwining operators as conjectured via Verlinde's formula \cite{CR1, CR2}.

Ultimatively one wants to understand also the vertex tensor structure of these categories, however even the existence is open, except for ordinary modules \cite{CHY}.


\subsection{Free field realizations and logarithmic modules}

Free-field realizations of affine vertex algebras and $\mathcal W$--algebras provide a good way for understanding their representation theories. It seems that each non-rational simple affine vertex algebra $L_k({\mathfrak g}) $ admits relaxed and logarithmic modules. 

But it is still  a hard task to prove this using the description of $L_k(\g)$ as a simple quotient of the universal affine vertex algebras $V^k({\mathfrak g})$. 

One simple, but very illustrative example is the admissible  affine vertex algebra $L_{-4/3}(\mathfrak{sl}_2)$. The existence of logarithmic modules was predicted by M. Gaberdiel in \cite{Gab-01}, while   the vertex algebraic construction of these logarithmic modules was obtained by combining the free-field realization from \cite{A-2005}, together with a (then) new method of a  construction of logarithmic representations from \cite{AdM-2009}.
An immediate corollary of the free field realization of $L_{-4/3}(\mathfrak{sl}_2)$ is that its Heisenberg coset is the $\mathcal M(3)$-singlet algebra (cf. \cite[Theorem 5.2]{A-2005}). This is useful, since rigid tensor categories of the singlet algebras including all logarithmic modules are understood \cite{CMY, CKMY} and the corresponding structure is inherited by  
$L_{-4/3}(\mathfrak{sl}_2)$ via vertex algebra extension theory \cite{CKM, CMY2}.  
 The free field realization of $L_{-4/3}(\mathfrak{sl}_2)$ generalizes to subregular $\mathcal W$-algebras of $\mathfrak{sl}_n$ at level $k = - n + \frac{n}{n+1}$ \cite{CRW, ACGY-2020}  and to many more including $L_{-3/2}(\mathfrak{sl}_3)$ \cite{A-2016, C-2018}. Another benefit of some free field realizations is that they indicate correspondences (called logarithmic Kazhdan-Lusztig correspondences) to quantum groups, see \cite{ACKR}  for the one that involves the subregular $\mathcal W$-algebras of $\mathfrak{sl}_n$ at level $k = - n + \frac{n}{n+1}$.
The logarithmic Kazhdan-Lusztig correspondence of $L_{-3/2}(\mathfrak{sl}_3)$ is thoroughly studied in the accompanying work \cite{CRR}. However all these free field realizations only apply to special $\mathcal W$-algebras at special levels and for the general case one needs a new idea:

In \cite{A-2019}, the affine vertex algebra $L_k(\mathfrak{sl}_2)$, $k \notin {\Z}_{\ge 0}$, is realized as a subalgebra of the vertex algebra $Vir_{c_k} \otimes \Pi(0)$, where $Vir_{c_k}$ is the simple Virasoro vertex algebra of central charge $c_k = 1 - \tfrac{6 (k+1) ^2}{k+2}$, and $\Pi(0)$ is  a half-lattice vertex algebra (an extension of a rank two Heisenberg algebra along a rank one lattice). This realization enabled a realization of all modules in $KL_k$, their intertwining operators, relaxed and logarithmic modules.
The construction generalizes straight forwardly to subregular $\mathcal W$-algebras of $\mathfrak{sl}_n$  and also $\mathfrak{so}_{2n+1}$
as well as principal $\mathcal W$-superalgebras of $\mathfrak{sl}_{n|1}$ and $\mathfrak{spo}_{2n|2}$ via duality \cite{CGN, CL1, CL2}. Note, that the subregular $\mathcal W$-algebra of $\mathfrak{sl}_2$ is nothing but the affine vertex algebra itself.
The subregular $\mathcal W$-algebra of $\mathfrak{sl}_3$ is also called the Bershadsky-Polyakov algebra and its representation theory at admissible levels has been studied using the free field realization in \cite{AKR-2020}.
 All these examples work, because the $\mathcal W$-(super)algebras allow for characterizations as intersections of kernels of screening charges on certain free field algebras. There is a technology to bring the screening charges into suitable form \cite{CGN} so that one can easily observe that the $\mathcal W$-(super)algebra embeds into a principal $\mathcal W$-algebra times a half lattice vertex algebra. At non-degenerate admissible levels the simple principal $\mathcal W$-algebra is rational and lisse \cite{Ara-rat, Ara-C2} and in particular its representation theory is of a simple form.


 

  
  
  

The aim of this paper is to lift this construction to encompass also the affine vertex algebra of $\mathfrak{sl}_3$. Our work is motivated from \cite{CGHL}. 
We explain in section \ref{realization-1} that the free field realization of $V^{k} (\mathfrak{sl}_3)$ can be brought into such a form that one has an obvious embedding of $V^{k} (\mathfrak{sl}_3)$ into $\BP \otimes \mathcal S \otimes \Pi(0) $,  where $\BP$ is the universal Bershadsky-Polyakov algebra (the subregular $\mathcal W$-algebra of $\mathfrak{sl}_3$)  at level $k$, $\mathcal S$ is the  $\beta \gamma$ vertex algebra  of rank one and $\Pi(0)$ is a lattice type vertex algebra.

 Let $\SBP$ be the simple quotient of $\BP$ and let $k\notin {\Z}_{\ge 0}$.  It has been studied at admissible level \cite{Ar, AKR-2020, FKR}.
 We can use the free field realization to prove the following results:

\begin{enumerate}
\item There is a  embedding of the simple affine vertex algebra  $L_k(\mathfrak{sl}_3) \hookrightarrow  \SBP \otimes  \mathcal S \otimes \Pi(0)$, and recall that $k \notin {\Z}_{\ge 0}$ (Theorem \ref{main}).
\item Each irreducible module $M$ from the category $KL_k$ is realized as a submodule of
$ L[x,y] \otimes \mathcal S \otimes \Pi(0)^{1/3}$, 
where $L[x,y]$ is an irreducible $\SBP$--module with $\dim L[x,y]_{top}< \infty$, and $\Pi(0)^{1/3}$ is defined in  (\ref{pi-ext}).

Conversely each $\SBP$-module $L[x,y]$ with finite dimensional top level  is obtained via quantum Hamiltonian reduction from a module in $KL_k$. 

These are the results of section \ref{sect-KL}.
\item 

 We construct in Sections  \ref{relaxed-constr-general} and  \ref{relaxed-fd} a family of irreducible relaxed $L_k(\mathfrak{sl}_3)$--module with finite-dimensional weight spaces  as  tensor product  \bea   R \otimes \Pi_{-1,-1} (\lambda_1, \lambda_2)\label{form-tensor-1-1} \eea where $R$ is an irreducible $\SBP$--module, and $\Pi_{-1,-1} (\lambda_1, \lambda_2)$ is a weight $\Pi(0)^{\otimes 2}$--module. Based on the results presented in \cite{AFR} and \cite{KRW-21} in the case of admissible levels, we expect that each irreducible relaxed modules with finite-dimensional weight spaces has the form (\ref{form-tensor-1-1}). For semi-relaxed modules  (cf. \cite{KRW-21}) we expect that all of them are realized as 
 \bea  R \otimes \mathcal S \otimes \Pi_{-1}(\lambda). \label{form-tensor-1-1}\eea
   
\item We also construct   modules with infinite-dimensional weight spaces $  R_M(\lambda) \otimes \Pi_{-1,-1}(\lambda_1, \lambda_2)$, for which we conjecture that they are irreducible (cf. Conjecture \ref{conj-1}). The proof of irreducibility is reduced to proving  that its top   component is an irreducible $\mathfrak{sl}_3$ module. Weight $\mathfrak{sl}_3$--modules of a similar type have appeared in the recent work  \cite{FLLZ-2016}, but so far we cannot identify them explicitly.

\item Finally in section \ref{log-bp} we construct logarithmic modules for $\SBP$. This together with our free field realization is then used in section \ref{log-sl} to get logarithmic modules for  $L_k(\mathfrak{sl}_3)$.
\end{enumerate}


\subsection{Outlook}

It is a combined effort of many researchers in the area to improve our understanding of the representation theory of affine vertex algebras at admissible levels. Past research has focussed on the simplest, but still rather non-trivial, example of $L_k(\mathfrak{sl}_2)$. The understanding is now good enough so that one can lift insights to higher rank, in particular $L_k(\mathfrak{sl}_3)$. This is done in the present work as well as the recent \cite{KRW-21, CRR}.

The free field realization of $L_k(\mathfrak{sl}_2)$ has been employed to obtain some fusion rules \cite{A-2019}, but conjecturally there are more \cite{CR2}. Our next goal is to study fusion rules of $L_k(\mathfrak{sl}_2)$, $L_k(\mathfrak{sl}_3)$ and $\SBP$. For this we will use free field realizations but also the coset realizations of these three algebras  \cite{ACL, ACF}. 

 Next one wants to have a full classification of modules. For $L_k(\mathfrak{sl}_2)$ that has been achieved in \cite{ACK} and it is realistic to hope that $\SBP$ will work in a similar fashion, but we expect $L_k(\mathfrak{sl}_3)$ to be substantially more difficult. The ultimate goal is then a full understanding of the representation category as a vertex tensor category. Here the main issues are existence and rigidity. Two problems that have been solved in the somehow simpler examples of the $\beta\gamma$-vertex algebra $\mathcal S$ \cite{AW} and $L_k(\mathfrak{gl}_{1|1})$ \cite{CMY3}.

Our realization of $L_k(\mathfrak{sl}_3)$ shows that previous results on the intertwining operators and  fusion rules  for the  vertex algebra $\mathcal S$ from \cite{AP-2019}   can be lifted  to $L_k(\mathfrak{sl}_3)$.  

Our construction also  lifts $\SBP$ to $L_k(\mathfrak{sl}_3)$. This means if we have a larger vertex algebra that contains $\SBP$ as a subalgebra, then our construction can be used to provide a vertex algebra that has  $L_k(\mathfrak{sl}_3)$. One nice example is nine $\beta\gamma$ vertex algebras that contain $\BP$ at critical level together with two copies of the affine vertex algebra of $\mathfrak{sl}_3$ at critical level. We expect that our lifting construction will give a realization of $L_{-3}(\mathfrak{e}_6)$. This example and related ones will be investigated further. 

\subsection{Organization}

Both $\BP$ and $V^k(\mathfrak{sl}_3)$ can be realized as subalgebras of certain free field algebras. These embeddings are characterized by certain screening operators. In section \ref{realization-1} we explain how to massage the free field realization of $V^k(\mathfrak{sl}_3)$ so that it becomes apparent that one has an embedding of $V^k(\mathfrak{sl}_3)$ in $\BP \otimes \mathcal S \otimes \Pi(0)$. Recall that $\mathcal S$ is the $\beta\gamma$ vertex algebra and $\Pi(0)$ is an extension of a rank two Heisenberg vertex algebra along a rank one lattice. We call this a half lattice VOA. Also note that $\mathcal S$ embeds itself in a copy of $\Pi(0)$.  
It is important for our purposes that we can express the fields of $V^k(\mathfrak{sl}_3)$ explicitely in terms of those of $\BP \otimes \mathcal S \otimes \Pi(0)$.

The explicit free field realization tells us that the free field algebra $S \otimes \Pi(0) \subset \Pi(0)^{\otimes 2}$ contains an affine vertex algebra $V^k(\mathfrak{b})$ for a certain solvable subalgebra $\mathfrak{b} \subset\mathfrak{sl}_3$ (it is not quite a Borel subalgebra). 
In Section \ref{pi0-voa} we study modules of $ \Pi(0)^{\otimes 2}$, list their characters and most imortantly study the action of $\widehat{\mathfrak b}$ on modules. 

The extended affine Weyl group of $\widehat{\mathfrak{sl}_3}$ induces automorphisms on the affine Lie algebra, the spectral flow automorphisms. In section \ref{spec-flow} we explain how spectral flow automorphisms of $\BP$ and $\Pi(0)^{\otimes 2}$ combine into the action of spectral flow on the image of $V^k(\mathfrak{sl}_3)$ in $\BP \otimes \Pi(0)^{\otimes 2}$. 

In section \ref{rel-simpleVOA} we use the results of the previous sections together with explicit computations to show that $L_k(\mathfrak{sl}_3)$ embeds into $\SBP \otimes S \otimes \Pi(0)$ if and only if $k$ is not a non-negative integer. This tells us that an $\SBP \otimes S \otimes \Pi(0)$-module is automatically also a module for $L_k(\mathfrak{sl}_3)$. 

Next we study singular vectors for the affine subalgebras  of $\BP$ and $\Pi(0)^{\otimes 2}$ and in particular use them to show that every ordinary module, i.e. every module in $KL_k(\mathfrak{sl}_3)$ appears as a submodule of $\BP$ and $\Pi(0)^{\otimes 2}$, see  section \ref{sect-KL}.

We then turn in section \ref{relaxed-constr-general} to relaxed modules. We show that simple modules of the free field algebra are almost simple modules for
$L_k(\mathfrak{sl}_3)$. Almost simple means that every submodule intersects the top level subspace non-trivially. In particular we then show that
such an   $L_k(\mathfrak{sl}_3)$ is simple if and only if the top level is simple as a $\mathfrak{sl}_3$-module.

If we start with a simple $\SBP$-module $L$ whose top level is infinite dimensional, then $L$  times the appropriate $\Pi(0)^{\otimes 2}$-module has as top-level an $\mathfrak{sl}_3$-module whose weight spaces are of infinite dimension, i.e. they are not of Gelfand-Tsetlin type. We determine the explicit action of 
$\mathfrak{sl}_3$ on these modules in section \ref{top}. This leads to a Conjecture on their simplicity.  We don't see an isomorphism to the modules constructed by Futorny, Liu, Lu and Zhao
\cite{FLLZ-2016} and so we leave it as an open question whether our modules are new. 

On the other hand, if we start with a simple $\SBP$-module $L$ whose top level is finite dimensional, then $L$  times the appropriate $\Pi(0)^{\otimes 2}$-module has as top-level an $\mathfrak{sl}_3$-module whose weight spaces are of finite dimension, i.e. they are  of Gelfand-Tsetlin type. In this case, we obtain relaxed modules with finite dimensional weight spaces and we can determine a condition on the weights that implies simplicity. This is the content of section \ref{relaxed-fd}.

Finally, we use the screening charges to deform the action of $L_k(\mathfrak{sl}_3)$ on modules of the free field algebra. This yields logarithmic modules: indecomposable modules with nilpotent action of the Virasoro zero-mode. This is done in section \ref{log-sl} and as a preparation we need to do a similar construction of logarithmic modules for $\SBP$, which is interesting in its own right, see section \ref{log-bp}.

\section{ Realization of the universal affine VOA $V^k(\mathfrak{sl}_3)$  }
 \label{realization-1}

The Bershadsky-Polyakov vertex algebra $\BP$ \cite{Bershadsky-1991, Polyakov-1990} is the $\mathcal{W}$-algebra obtained from the affine vertex algebra of  $\mathfrak{sl}_3$ at level $k$ via quantum Hamiltonian reduction associated to a minimal nilpotent elemenr $f$ of $\mathfrak{sl}_3$.  It is freely and strongly generated by four fields $J(z), G^\pm(z), L^\mathrm{BP}(z)$. For non-critical $k$, that is $k\neq  -3$, the operator products are
\begin{align*}
&J(z)J(w) \sim \frac{2k+3}{3(z-w)^2},\quad
G^\pm(z)G^\pm(w) \sim 0,\\
&J(z)G^\pm(w) \sim \frac{\pm G^\pm(w)}{z-w},\\
&L^\mathrm{BP}(z)L^\mathrm{BP}(w) \sim -\frac{(2k+3)(3k+1)}{2(k+3)(z-w)^4} + \frac{2L^\mathrm{BP}(w)}{(z-w)^2} + \frac{\partial L^\mathrm{BP}(w)}{z-w},\\
&L^\mathrm{BP}(z)G^\pm(w) \sim \frac{3G^\pm(w)}{2(z-w)^2} + \frac{\partial G^\pm(w)}{z-w},\\
&L^\mathrm{BP}(z)J(w) \sim \frac{J(w)}{(z-w)^2} + \frac{\partial J(w)}{z-w},\\
&G^+(z)G^-(w) \sim \frac{(k+1)(2k+3)}{(z-w)^3} + \frac{3(k+1)J(w)}{(z-w)^2}\\
&\hspace{40mm}+\frac{1}{z-w}\left( 3:J(w)^2:+\frac{3(k+1)}{2}\partial J(w)-(k+3)L^\mathrm{BP}(w) \right).
\end{align*}
In the case $k = -3$, we need to replace the Virasoro field $L^\mathrm{BP}(z)$ with the commutative field $T^\mathrm{BP}(z)$ in the set of the generators and OPE relations above by $(k+3)L^\mathrm{BP}(z) \mapsto T^\mathrm{BP}(z)$. 
The central charge of  $\BP$ is 
\[
c_k = -\frac{(2k+3)(3k+1)}{(k+3)}. 
\]
Note that $L^\mathrm{BP}(z) + \frac{1}{2} \partial J(z)$ also defines a Virasoro field on $\BP$, but now with the central charge
\begin{align*}
\tilde c_k = -\frac{4(k+1)(2k+3)}{k+3},
\end{align*}
which gives a $\Z_{\geq0}$-grading on $\BP$ such that conformal weights of $J(z)$, $G^+(z)$, $G^-(z)$ are $1, 1, 2$ respectively.
\smallskip

Let $\mathcal{S}(n)$ be the $n$-fold tensor product of the $\beta\gamma$ vertex algebra, $\pi^k$ the Heisenberg vertex algebra associated to the Cartan subalgebra $\mathfrak{h}$ of $\mathfrak{sl}_3$ at level $k$ and $\pi^k_\lambda$ the Fock module of $\pi^k$ with the highest weight $\lambda\in\mathfrak{h}^*$. Fix simple roots $\alpha_1, \alpha_2 \in \mathfrak{h}^*$ of $\mathfrak{sl}_3$. Recall Wakimoto modules of $V^k(\mathfrak{sl}_3)$ in \cite{Wakimoto-1986, FF-1990, Frenkel-2005}. Let $N_+$ be the Lie group corresponding to the upper nilpotent subalgebra $\mathfrak{n}_+$ of $\mathfrak{sl}_3$. Then the natural representation $\mathfrak{sl}_3 \hookrightarrow \mathfrak{gl}_3 = \operatorname{End}(\C^3)$ implies a matrix represeantation of $N_+$. We take an affine coordinate system on $N_+$ by
\begin{align*}
N_+ =\left\{
\left(
\begin{array}{ccc}
1 & -x_1 & -x_3\\
0 & 1 & -x_2\\
0 & 0 & 1
\end{array}
\right)
\:\middle|\: x_i \in \C
\right\}.
\end{align*}
Using the formulae (1.4) and Theorem 5.1 in \cite{Frenkel-2005}, $V^k(\mathfrak{sl}_3)$ has a Wakimoto free field realization
\begin{align*}
V^k(\mathfrak{sl}_3) \hookrightarrow \mathcal{S}(3)\otimes\pi^{k+3}
\end{align*}
defined by
\begin{equation*}
\begin{array}{cl}
e_1& \mapsto \beta_1 - :\gamma_2\beta_3:,\\[2mm]
e_2& \mapsto \beta_2,\\[2mm]
e_3& \mapsto \beta_3,\\[2mm]
h_1& \mapsto -2:\gamma_1\beta_1:+:\gamma_2\beta_2:-:\gamma_3\beta_3:+\alpha_1,\\[2mm]
h_2& \mapsto\ :\gamma_1\beta_1:-2:\gamma_2\beta_2:-:\gamma_3\beta_3:+\alpha_2,\\[2mm]
f_1& \mapsto -:\gamma_1^2\beta_1:-:\gamma_3\beta_2:+(k+1)\partial \gamma_1+\alpha_1\gamma_1,\\[2mm]
f_2& \mapsto\ :(\gamma_1\gamma_2+\gamma_3)\beta_1:-:\gamma_2^2\beta_2:-:\gamma_2\gamma_3\beta_3:+k\partial \gamma_2+\alpha_2\gamma_2,\\[2mm]
f_3& \mapsto -:\gamma_1(\gamma_1\gamma_2+\gamma_3)\beta_1-:\gamma_2\gamma_3\beta_2:-:\gamma_3^2\beta_3:+k\partial\gamma_3\\[2mm]
&\qquad+(k+1)(\partial\gamma_1)\gamma_2+:\alpha_1(\gamma_1\gamma_2+\gamma_3):+\alpha_2\gamma_3,
\end{array}
\end{equation*}
where $(\beta_i, \gamma_i)$ is the $\beta\gamma$-pair of the $i$-th $\beta\gamma$-system in $\mathcal{S}(3)$, $h_1 = e_{1,1}-e_{2,2}$, $h_2 = e_{2,2}-e_{3,3}$, $e_1 = e_{1,2}$, $e_2 = e_{2,3}$, $e_3 = e_{1,3}$, $f_1 = e_{2,1}$, $f_2 = e_{3,2}$, $f_3 = e_{3,1}$ under the natural representation $\mathfrak{sl}_3 \hookrightarrow \mathfrak{gl}_3 = \operatorname{End}(\C^3)$, and $\{ e_{i, j}\}_{i, j = 1}^3$ denotes the standard basis of  $\mathfrak{gl}_3$. Thanks to Proposition 8.2 in \cite{Frenkel-2005}, the image of the map above is characterized by the screening operators if $k$ is generic:
\begin{align*}
V^k(\mathfrak{sl}_3) &\cong \Ker \int \beta_1(z)\mathrm{e}^{-\frac{\alpha_1}{k+3}}(z)dz\, 
\cap\, \Ker \int (\gamma_1\beta_3-\beta_2)(z)\mathrm{e}^{-\frac{\alpha_2}{k+3}}(z)dz.
\end{align*}
where $\mathrm{e}^\lambda(z)$ is an intertwining operator between Heisenberg Fock modules defined by
\begin{align*}
\mathrm{e}^\lambda(z) = s_\lambda z^{\lambda{(0)}}\exp\left(-\sum_{n<0}\lambda{(n)}\frac{z^{-n}}{n}\right)\exp\left(-\sum_{n>0}\lambda{(n)}\frac{z^{-n}}{n}\right)
\end{align*}
and $s_\lambda$ is the shift operator mapping the highest weight vector to the highest weight vector by shifting $\lambda$, and commuting with all $\mu(n)$, $n\neq0$ for $\mu \in \mathfrak{h}^*$. Let $Z$ be the lattice spanned by $x_i, y_i$ for $i=1, \ldots, n$ with $(x_i|x_j)=\delta_{i,j}=-(y_i|y_j)$ and $(x_i|y_j)=0$, and $\mathcal{V}(n)$ be the vertex subalgebra of $V_{Z}$ generated by $x_i, y_i, \mathrm{e}^{x_i+y_i}, \mathrm{e}^{-x_i-y_i}$ for $i=1,\ldots,n$. The bosonization of $\mathcal{S}(n)$ \cite{FMS-1986} is the embedding
\begin{align*}
\mathcal{S}(n) \hookrightarrow \mathcal{V}(n),\quad
\beta_i \mapsto \mathrm{e}^{x_i+y_i},\quad
\gamma_i\mapsto-:x_i\mathrm{e}^{-x_i-y_i}:,
\end{align*}
whose image is characterized by the screening operators
\begin{align*}
\mathcal{S}(n) \cong \bigcap_{i=1}^n\Ker\left(\int \mathrm{e}^{x_i}(z)dz\colon\mathcal{V}(n)\rightarrow V_Z\right).
\end{align*}
Using the bosonization of $\beta_3, \gamma_3$
\begin{align*}
& \mathcal{S}(1) \hookrightarrow \mathcal{V}(1),\quad
\beta_3 \mapsto \mathrm{e}^{x+y},\quad
\gamma_3 \mapsto -\NO{x\mathrm{e}^{-x-y}},
\end{align*}
we have
\begin{align*}
\rho_0 \colon V^k(\mathfrak{sl}_3) \hookrightarrow \mathcal{S}(2) \otimes \mathcal{V}(1) \otimes \pi^{k+3}
\end{align*}
such that, for generic $k$,
\begin{equation*}
 \operatorname{Im}\rho_0 =  
\Ker \int \beta_1(z)\mathrm{e}^{-\frac{\alpha_1}{k+3}}(z) dz\,
\cap\,
\Ker \int (\gamma_1\mathrm{e}^{x+y} -\beta_2)(z)\mathrm{e}^{-\frac{\alpha_2}{k+3}}(z) dz
\,  \cap\,
\Ker \int \mathrm{e}^x(z) dz.
\end{equation*}
Set
\begin{align*}
&\widetilde{\beta} = \beta_1,\quad
\widetilde{\gamma} = \gamma_1-\beta_2\mathrm{e}^{-x-y},\quad
\beta = \beta_2,\quad
\gamma = \gamma_2-\beta_1\mathrm{e}^{-x-y},\\
&\widetilde{\alpha}_1 = \alpha_1,\quad
\widetilde{\alpha}_2 = \alpha_2-(k+3)(x+y),\quad
c = x+y,\\
&d=-\frac{2k+3}{3}x-\frac{2k+9}{3}y-2\beta_1\beta_2\mathrm{e}^{-x-y}+\frac{2}{3}(\alpha_1+2\alpha_2).
\end{align*}
Let $\widetilde{\mathcal{S}}$ be a vertex algebra generated by $\widetilde{\beta}, \widetilde{\gamma}$, $\mathcal{S}$ a vertex algebra generated by $\beta, \gamma$, $\widetilde{\pi}^{k+3}$ a vertex algebra generated by $\widetilde{\alpha}_1, \widetilde{\alpha}_2$ and $\Pi(0)$ a vertex algebra generated by $c, d, \mathrm{e}^{\pm c}$. Then
\begin{align*}
\mathcal{S}(2) \otimes \mathcal{V}(1) \otimes \pi^{k+3} = \widetilde{\mathcal{S}} \otimes \mathcal{S} \otimes \widetilde{\pi}^{k+3} \otimes \Pi(0),
\end{align*}
and for generic $k$,
\begin{equation}\label{eq:scr-rho0}
 \operatorname{Im}\rho_0 = 
 \Ker \int \widetilde{\beta}(z)\mathrm{e}^{-\frac{\widetilde{\alpha}_1}{k+3}}(z) dz\,
\cap\,
\Ker \int \widetilde{\gamma}(z)\mathrm{e}^{-\frac{\widetilde{\alpha}_2}{k+3}}(z) dz\,
 \cap\,
\Ker \int \mathrm{e}^x(z) dz.
\end{equation}
From \cite{FS-2004}, $\BP$ may be defined as a vertex subalgebra of $\widetilde{\mathcal{S}}\otimes\widetilde{\pi}^{k+3}$ by
\begin{align*}
&\rho_1 \colon \BP \hookrightarrow \widetilde{\mathcal{S}}\otimes\widetilde{\pi}^{k+3},\\
&J \mapsto \displaystyle -\frac{1}{3}\widetilde{\alpha}_1 + \frac{1}{3}\widetilde{\alpha}_2 + \NO{\widetilde{\beta}\widetilde{\gamma}},\\
&G^+ \mapsto \widetilde{\alpha}_1\widetilde{\gamma} - \NO{\widetilde{\beta}\widetilde{\gamma}^2} +(k+1)\partial\widetilde{\gamma},\quad
G^- \mapsto \widetilde{\alpha}_2\widetilde{\beta} + \NO{\widetilde{\beta}^2\widetilde{\gamma}} +(k+1)\partial\widetilde{\beta},\\
&L^{\mathrm{BP}} \mapsto  \frac{1}{k+3}\left(\frac{1}{3}\NO{(\widetilde{\alpha}_1^2+\widetilde{\alpha}_1\widetilde{\alpha}_2+\widetilde{\alpha}_2^2)}+\frac{k+1}{2}(\partial\widetilde{\alpha}_1+\partial\widetilde{\alpha}_2)\right)
+\frac{1}{2}\NO{\widetilde{\beta} \partial\widetilde{\gamma}}-\frac{1}{2}\NO{(\partial\widetilde{\beta}) \widetilde{\gamma}},
\end{align*}
such that, for generic $k$,
\begin{align*}
\operatorname{Im}\rho_1 = 
\Ker \int \widetilde{\beta}(z)\mathrm{e}^{-\frac{\widetilde{\alpha}_1}{k+3}}(z) dz\,
\cap\,
\Ker \int \widetilde{\gamma}(z)\mathrm{e}^{-\frac{\widetilde{\alpha}_2}{k+3}}(z) dz.
\end{align*}
Since $\rho_1$ consists of the Miura map of $\BP$ and the Wakimoto free field realization of $V^{k+1}(\mathfrak{sl}_2)$, $\rho_1$ is injective for all $k$, where we will replace $(k+3)L^{\mathrm{BP}}$ by $T^{\mathrm{BP}}$ in case that $k=-3$. Thus, for generic $k$, by forgetting the screening operator $\int \mathrm{e}^x(z) dz$ in \eqref{eq:scr-rho0}, we obtain an embedding
\begin{align}\label{eq:Phi_0-emb}
\Phi_0 \colon V^k(\mathfrak{sl}_3) \rightarrow \BP \otimes \mathcal{S} \otimes \Pi(0)
\end{align}
such that
\begin{align}\label{eq:rho_012}
\rho_0 = (\rho_1 \otimes \operatorname{id}) \circ \Phi_0.
\end{align}
Then the map $\Phi_0$ is well-defined for all $k$. Indeed, $\Phi_0$ is defined by
\begin{align*}
&e_1 \mapsto -\gamma\mathrm{e}^c,\quad
e_2 \mapsto \beta,\quad
e_3 \mapsto \mathrm{e}^c,\\
&h_1 \mapsto \displaystyle -2J+\NO{\gamma\beta}-\frac{2k+9}{6}c+\frac{1}{2}d,\quad
h_2 \mapsto \displaystyle J-2\NO{\gamma\beta}+\frac{4k+9}{6}c+\frac{1}{2}d,
\end{align*}
\begin{align*}
&f_1 \mapsto \displaystyle G^+-\NO{\left(2J+\frac{8k+9}{6}c-\frac{1}{2}d\right)\beta\mathrm{e}^{-c}}+(k+1)(\partial\beta)\mathrm{e}^{-c},\\
&f_2 \mapsto \displaystyle G^-\mathrm{e}^{-c}+\left(J+\frac{4k+9}{6}c+\frac{1}{2}d\right)\gamma+k\partial \gamma-\NO{\gamma^2\beta},\\
&f_3 \mapsto G^+\gamma-G^-\beta\mathrm{e}^{-2c}
+\NO{
\Bigl\{
(k+3)L^\mathrm{BP}
-J\left(J+\frac{2k-9}{6}c-\frac{1}{2}d\right)\\
&\qquad+\frac{k+1}{2}\partial\left(J+\frac{2}{3}k c-d\right)
-2\gamma\beta\left(J+\frac{8k+9}{12}c-\frac{1}{4}d\right)\\
&\qquad+(k+1)\gamma(\partial \beta)
-\frac{4k^2-18k-27}{36}c^2
+\frac{k}{3}cd
-\frac{1}{4}d^2
\Bigr\}\mathrm{e}^{-c}
}.
\end{align*}
Since $\rho_0$ and $\rho_1$ are injective, so is $\Phi_0$ for all $k$ by \eqref{eq:rho_012}. The Sugawara construction gives a Virasoro field $L_{sug}$ on $V^k(\mathfrak{sl}_3)$ with the central charge $8k/(k+3)$ if $k \neq -3$. Then the image of $L_{sug}$ is 
\bea \label{eq:Sugawara-new}
L_{sug} \mapsto L^\mathrm{BP} + \frac{1}{2}\partial J + \NO{(\partial\gamma)\beta} + \frac{1}{2}\NO{cd} +\frac{k}{3}\partial c -\frac{1}{2} \partial d .
\eea
Consider the bosonization of $\mathcal{S}\otimes \Pi(0)$:
\begin{align*}
&\mathcal{S} \otimes \Pi(0) \hookrightarrow \Pi(0)^{\otimes 2},\\
&\beta \mapsto \mathrm{e}^{c_1},\quad
\gamma \mapsto -\frac{1}{2}\NO{(c_1+d_1)\mathrm{e}^{-c_1}},\quad
c \mapsto c_2,\quad
d \mapsto d_2,
\end{align*}
where $\Pi(0)^{\otimes 2}$ is a vertex algebra generated by $c_1, c_2, d_1, d_2, \mathrm{e}^{\pm c_1}, \mathrm{e}^{\pm c_2}$ with OPE relations
\begin{align*}
c_i(z) d_j(w) \sim \frac{2\delta_{i, j}}{(z-w)^2},\quad
c_i(z)c_j(w) \sim 0 \sim d_i(z)d_j(w).
\end{align*}
Then we have an embedding
\begin{align}\label{eq:Phi_1-emb}
\Phi_1 \colon V^k(\mathfrak{sl}_3) \hookrightarrow \BP \otimes \Pi(0)^{\otimes 2}
\end{align}
defined by
\begin{align*}
&e_1 \mapsto \frac{1}{2}:(c_1+d_1)\mathrm{e}^{-c_1+c_2}:,\quad
e_2 \mapsto \mathrm{e}^{c_1},\quad
e_3 \mapsto \mathrm{e}^{c_2},\\
&h_1 \mapsto -2J+\frac{1}{2}c_1-\frac{1}{2}d_1-\frac{2k+9}{6}c_2+\frac{1}{2}d_2,\\
&h_2 \mapsto J-c_1+d_1+\frac{4k+9}{6}c_2+\frac{1}{2}d_2,\\
&f_1 \mapsto G^+ - :\left( 2J -(k+1)c_1+\frac{8k+9}{6}c_2-\frac{1}{2}d_2 \right)\mathrm{e}^{c_1-c_2}:,\\
&f_2 \mapsto G^- \mathrm{e}^{-c_2}-\frac{k+1}{2}:(\partial c_1 + \partial d_1)\mathrm{e}^{-c_1}:\\
&\qquad - \frac{1}{2}:\left( J -\frac{2k+3}{2}c_1+\frac{1}{2}d_1+\frac{4k+9}{6}c_2+\frac{1}{2}d_2  \right)(c_1+d_1)\mathrm{e}^{-c_1}:,\\
&f_3 \mapsto -\frac{1}{2} :G^+(c_1+d_1)\mathrm{e}^{-c_1}:-G^-\mathrm{e}^{c_1-2c_2}\\
&\qquad + :\biggl( (k+3)L^\mathrm{BP} + \frac{k+1}{2}\partial \left( J+c_1+\frac{2}{3}k c_2-d_2 \right)\biggr)\mathrm{e}^{-c_2}:\\
&\qquad +:\biggl\{ -J \left( J+c_1-d_1+\frac{2k-9}{6}c_2-\frac{1}{2}d_2 \right)
-\frac{1}{12}(c_1-d_1) \left( (8k+9)c_2 - 3 d_2 \right)\\
&\qquad  -\frac{k+2}{2}c_1 d_1 -\frac{4k^2-18k-27}{36} (c_2)^2 +\frac{k}{3}c_2 d_2 -\frac{1}{4}(d_2)^2 \biggr\}\mathrm{e}^{-c_2}:.
\end{align*}
The image of $L_{sug}$ is
\bea  
L_{sug} \mapsto L^\mathrm{BP} +\frac{1}{2}\partial J +\frac{k}{3}\partial c_2 -\frac{1}{2} \partial d_1 -\frac{1}{2} \partial d_2 + \frac{1}{2}:(c_1 d_1 + c_2 d_2):. \nonumber \\  \label{eq:Sugawara} 
\eea

\section{The vertex algebra $\Pi(0) ^{\otimes 2}$ and its weight modules}
\label{pi0-voa}
We choose the Virasoro vector in  $\Pi(0) ^{\otimes 2}$ to be

$$L^ {\Pi}  =  \frac{k}{3}\partial c_2 -\frac{1}{2} \partial d_1 -\frac{1}{2} \partial d_2 + \frac{1}{2}:(c_1 d_1 + c_2 d_2): $$
It has central charge $c_{\Pi}  = 4 - 12  (- \frac{ 2 k}{3})   = 4   +  8 k $.
Define the  $\Pi(0)  ^{\otimes 2}$ --module  
\begin{equation}
\Pi_{r_1, r_2}(  \lambda_1, \lambda_2) := \Pi(0) ^{\otimes 2} e^{ r_1 \frac{d _1}{2}  + r_2 \frac{d _2}{2} + \lambda_1 c_1 + \lambda_2 c_2}
\end{equation}
  where $r_1, r_2 \in {\Z}$, $\lambda_1, \lambda_2 \in {\C}$. Note that $\Pi_{r_1, r_2}(  \lambda_1, \lambda_2) \cong \Pi_{r_1, r_2}(  \lambda_1+n_1, \lambda_2+n_2)$ for $n_1, n_2 \in \mathbb Z$. 
  We have
 \begin{equation}\nonumber
 \begin{split}
 L^{\Pi} (0)  \  e^{ r_1 \frac{d _1}{2}  + r_2 \frac{d _2}{2} + \lambda_1 c_1 + \lambda_2 c_2} 
 &=  ( - \frac{k}{3} r_2 +  (1+ r_1) \lambda_1 +  (1+r_2) \lambda_2)    \  e^{ r_1 \frac{d _1}{2}  + r_2 \frac{d _2}{2} + \lambda_1 c_1 + \lambda_2 c_2} 
 \end{split}
 \end{equation}
 This implies that for $r_1 = r_2 = -1$  and all $\lambda_1, \lambda_2 \in {\C}$ we have
  \begin{equation}
 L^{\Pi} (0)  \  e^{  - \frac{d _1}{2}  - \frac{d _2}{2} + \lambda_1 c_1 + \lambda_2 c_2}= \frac{k}{3} e^{  - \frac{d _1}{2}  - \frac{d _2}{2} + \lambda_1 c_1 + \lambda_2 c_2}. 
  \end{equation}
  
 Let $\mu_1 = \frac{1}{2}c_1-\frac{1}{2}d_1-\frac{2k+9}{6}c_2+\frac{1}{2}d_2$ and $\mu_2=
-c_1+d_1+\frac{4k+9}{6}c_2+\frac{1}{2}d_2$.
  Define the character of a   $\Pi(0) ^{\otimes 2}$-module $M$ to be 
\[
\text{ch}[M] : = \text{tr}_{M}\left(q^{L^\Pi(0)-\frac{c_\Pi}{24}} z_1^{\mu_1(0)}z_2^{\mu_2(0)} 
 \right)
\]
for formal variables $q, z_1, z_2$.
In particular, the graded trace of $\Pi_{-1, -1}(  \lambda_1, \lambda_2) $ is 
\begin{equation}
\begin{split}
\text{ch}[\Pi_{-1, -1}(  \lambda_1, \lambda_2)] &= q^{-\frac{k}{3}+\frac{k}{3}} z_1^{-\frac{1}{2}+\frac{2k+9}{6}-\lambda_1+\lambda_2} z_2^{1-\frac{4k+9}{6}+2\lambda_1+\lambda_2}  
  \frac{\sum\limits_{n_1, n_2 \in \mathbb Z} z_1^{n_2-n_1} z_2^{2n_1+n_2} }{\eta(q)^4}\\
&=   z_1^{1+\frac{k}{3}-\lambda_1+\lambda_2} z_2^{-\frac{1}{2}-\frac{2k}{3}+2\lambda_1+\lambda_2}  \frac{\delta(z_1^{-1}z_2^2, z_1z_2)}{\eta(q)^4},
\end{split}
\end{equation}
where
\[
\delta(z_1, z_2) = \sum\limits_{n_1, n_2 \in \mathbb Z} z_1^{2n_1}z_2^{2n_2}
\]
the formal delta distribution in two variables, in the sense that it satisfies
\begin{equation}
f(z_1, z_2)\delta\left(\frac{z_1}{w_1}, \frac{z_2}{w_2}\right)\frac{1}{w_1w_2} = f(w_1, w_2)\delta\left(\frac{z_1}{w_1}, \frac{z_2}{w_2}\right)\frac{1}{w_1w_2} 
\end{equation}
for any formal power series $f(z_1, z_2)$.

\subsection{Partitions}

 Recall that a partition is a finite sequence of positive integers $\mu = (\mu_1, \mu_2, \dots, \mu_{\ell})$ of length $\ell = \ell(\mu) \in \Z_{> 0}$ satisfying
\begin{equation}
	\mu_1 \ge \mu _2 \ge \dots \ge \mu_{\ell}.
\end{equation}
The weight of the partition $\mu$ is defined to be $\abs{\mu} = \mu_1 + \mu_2 + \dots + \mu_{\ell}$.  Let $\mathcal P $ denote the set of all partitions.

Given a partition $\mu \in \mathcal P$ of length $\ell$ and an element $A$ of a vertex algebra, we introduce (whenever it makes sense) the convenient notation
\begin{equation} \label{eq:modepartitions}
	\begin{aligned}
		A_{\mu} &= A_{\mu_{\ell}} \cdots A_{\mu_2} A_{\mu_1}, &
		A_{-\mu} &= A_{-\mu_1} A_{-\mu_2} \cdots A_{-\mu_{\ell}}, \\
		A_{\mu -1} &= A_{\mu_{\ell}-1} \cdots A_{\mu_2 -1} A_{\mu_1-1}, &
		 	\end{aligned}
\end{equation}

\subsection{$\Pi_{0,0}(\lambda_1, \lambda_2)$ as a $\widehat{\mathfrak b}$--module}

Let $\widehat{\mathfrak b}$ be the Lie algebra generated by  $e_1(n), e_2(n), e_3(n) ,  \bar h (n)$, $n \in {\Z}$ where
$\bar h  (n) = h_1(n) + 2 h_2 (n)$.

Note that $\widehat{\mathfrak b}$ is a Lie subalgebra of $\widehat{\mathfrak{sl}}_3$.

Using realization we see that $U ( \widehat{\mathfrak b}  ). {\bf 1}$ is a vertex subalgebra of $\mathcal S \otimes \Pi(0) \subset \Pi(0) ^{\otimes 2}$ generated by

\bea
e_1 & = &   -\gamma\mathrm{e}^{c_2}  = \frac{1}{2} (c_1 + d_1) e^{-c_1 + c_2} \nonumber \\
e_2 &=& \beta = e^{c_1} \nonumber \\
e_3 &=& e^{c_2} \nonumber \\
  \bar h & =& h_1 + 2 h_2 = - 3 \NO{\gamma \beta} + \frac{2k+3}{2} c_2 + \frac{3}{2}   d_2 \nonumber \\ &=& \frac{2k+3}{2} c_2 + \frac{3}{2} (-c_1 + d_1 + d_2 ) \nonumber  
\eea
Let $\bar c_1 = -c_1 + c_2$. 
Set ${\bar \ell}_i = \ell_i + \lambda_i$, for  $i=1,2$ and $\ell_i \in {\Z}$.
Then the set:
$$\mathcal B_{\Pi_{0,0} (\lambda_1, \lambda_2)} = \{ (d_1)_{-\mu_1} (d_2)_{-\mu_2} (\bar c_1)_{-\nu_1}  (c_2)_{-\nu_2} e^{\bar \ell_1 c_1 + \bar \ell_2 c_2}  \vert  \ \mu_1, \mu_2, \nu_1, \nu_2 \in \mathcal P, \ell_1, \ell_2 \in {\Z}\},
$$
is a  basis of $\Pi(0) ^{\otimes 2}$.


By using direct calculation and the relations
 \begin{equation}
 \begin{split}   
 [e_1 (n), \bar c_1(m) ] &= - m e^{\bar c}_{n+m-1} \\ \label{vazna} 
 [e_2(n), d_1(m)] &= - e^{c_1}_{n+m} \\
  [e_3(n), d_2(m)] &= - e^{c_2}_{n+m} \\
 \end{split}
 \end{equation}
 we get the following important technical lemma:

\begin{lemma}  We have
\bea
(e_2)_{\mu_1-1}  (d_1)_{-\mu_1} (d_2)_{-\mu_2} (\bar c_1)_{-\nu_1}  (c_2)_{-\nu_2} e^{\bar \ell_1 c_1 + \bar \ell_2 c_2} &=& A_1  (d_2)_{-\mu_2} (\bar c_1)_{-\nu_1}  (c_2)_{-\nu_2} e^{(\bar \ell_1  + \ell(\mu_1)  )c_1 + \bar \ell_2 c_2} \nonumber  \\
(e_2)_{\bar \mu_1-1}  (d_1)_{-\mu_1} (d_2)_{-\mu_2} (\bar c_1)_{-\nu_1}  (c_2)_{-\nu_2} e^{\bar \ell_1 c_1 + \bar \ell_2 c_2} &=&  0  \quad \mbox{if} \ \mu_1 < \bar \mu_1 \nonumber \\
(e_3)_{\mu_2-1}   (d_2)_{-\mu_2} (\bar c_1)_{-\nu_1}  (c_2)_{-\nu_2} e^{\bar \ell_1 c_1 +\bar  \ell_2 c_2} &=& A_3  (\bar c_1)_{-\nu_1}  (c_2)_{-\nu_2} e^{\bar \ell_1    c_1 + ( \bar \ell_2 + \ell(\mu_2) ) c_2} \nonumber  \\
( e_3)_{\bar \mu_2-1}   (d_2)_{-\mu_2} (\bar c_1)_{-\nu_1}  (c_2)_{-\nu_2} e^{\bar \ell_1 c_1 +\bar  \ell_2 c_2} &=&  0  \quad \mbox{if} \  \mu_2 < \bar \mu_2 \nonumber \\
( \bar h)_{ \nu_2}     (\bar c_1)_{-\nu_1}  (c_2)_{-\nu_2} e^{\bar \ell_1 c_1 + \bar \ell_2 c_2} &=&   A_3 (\bar c_1)_{-\nu_1}    e^{\bar \ell_1 c_1 + \bar  \ell_2 c_2}  \nonumber \\
( \bar h)_{ \bar \nu_2}     (\bar c_1)_{-\nu_1}  (c_2)_{-\nu_2} e^{\bar \ell_1 c_1 + \bar \ell_2 c_2} &=&    0   \quad \mbox{if} \  \nu_2 < \bar \nu_2 \nonumber \\
 (e_1)_{\nu_1}  (\bar c_1)_{-\nu_1}    e^{\bar \ell_1 c_1 + \bar \ell_2 c_2} &=&  A_4 e^{(\bar \ell_1 -\ell(\nu_1) ) c_1 +  (\bar \ell_2 + \ell (\nu_1)) c_2} \nonumber \\
 (e_1)_{\bar \nu_1}  (\bar c_1)_{-\nu_1}    e^{\bar \ell_1 c_1 + \bar \ell_2 c_2} &=&  0 \quad \mbox{if } \ \nu_1 < \bar \nu_1, \nonumber \eea
 for some nonzero constants $A_i$, $i=1,2,3,4$.
\end{lemma}

As a consequence we get:
\begin{proposition}   \label{step-1}
For each $ w \in \Pi_{0,0} (\lambda_1, \lambda_2)$ there is $y \in   U(  \widehat{\mathfrak b})$ such that
$$ y. w \in \Pi_{0,0} (\lambda_1, \lambda_2) _{top} = {\C}[({ \Z} + {\lambda_1}) c_1 +( \Z + {\lambda_2}){c_2}]. $$ 
\end{proposition}
 \subsection{ $\Pi_{r_1, r_2}(  \lambda_1, \lambda_2)$ as  $U(  \widehat{\mathfrak b} )$--module }

  Consider again irreducible $\Pi(0)^{\otimes 2}$--module:
  $$\Pi_{r_1, r_2}(  \lambda_1, \lambda_2) := \Pi(0) ^{\otimes 2} e^{ r_1 d _1 / 2 + r_2 d _2  / 2 + \lambda_1 c_1 + \lambda_2 c_2}$$
  where $r_1, r_2 \in {\Z}$, $\lambda_1, \lambda_2 \in {\C}$.

  \begin{lemma}  \label{ired-relaxed-1} We have:
  $$ \Pi_{-1, -1}(  \lambda_1, \lambda_2)  = U(  \widehat{\mathfrak b}).   \Pi_{-1, -1}(  \lambda_1, \lambda_2)_{top}. $$
  \end{lemma}
 \begin{proof}
 The proof follows from the  fact that the set
 $$\mathcal B_{\Pi(\lambda_1, \lambda_2) } = \{ \bar h  _{-\mu} ( e_1 )_{-\nu_1} (e_2)_{-\nu_2}  (e_3)_{-\nu_3} e^{ (\lambda_1 + \ell_1) c_1 + (\lambda_2 + \ell_2) c_2}  \vert  \ \mu,  \nu_1, \nu_2, \nu_3 \in \mathcal P, \ell_1, \ell_2 \in {\Z}\}.
$$
is a basis of $\Pi_{-1,-1} (\lambda_1, \lambda_2)$.
Let us prove this claim 
We have the following natural basis of $\Pi_{-1, -1}(  \lambda_1, \lambda_2)$
 \bea && d_2 (-k_1) \cdots d_2(-k_i) d_1 (-l_1) \cdots d_1 (-l_j)  c_1 (-n_1) \cdots   c_1 (-n_r) \nonumber \\ &&      c_2 (-m_1) \cdots   c_2 (-m_s)  e^ {-\frac{1}{2} d_1 - \frac{1}{2} d_ 2 + (\lambda _1+ \ell_1) c_1 + (\lambda_2 + \ell_2)  c_2} \label{basis-1} \eea
where $\ell_1, \ell_2 \in {\Z}$ and 
$$ k_1 \ge \cdots \ge k_i \ge 1, l_1 \ge \cdots \ge l_j \ge 1,  n_1 \ge \cdots \ge n_r \ge 1, m_1 \ge \cdots \ge m_s \ge 1. $$
Since each element
 \bea && \bar h  (-k_1) \cdots \bar h  (-k_i) e_1 (-l_1) \cdots e_1 (-l_j)  e_2 (-n_1) \cdots   e_2 (-n_r) \nonumber \\ &&      e_3(-m_1) \cdots   e_3 (-m_s)  e^ {-\frac{1}{2} d_1 - \frac{1}{2} d_ 2 + (\lambda _1+ \ell_1 + j -r) c_1 + (\lambda_2 + \ell_2 -j -s)  c_2} \label{basis-2} \eea
 contains a  unique summand which is non-trivial scalar multiple of basis vector (\ref{basis-1}), we conclude that all vectors (\ref{basis-2}) with 
 $$ k_1 \ge \cdots \ge k_i \ge 1, l_1 \ge \cdots \ge l_j \ge 1,  n_1 \ge \cdots \ge n_r \ge 1, m_1 \ge \cdots \ge m_s \ge 1 $$
 form a new basis of  $\Pi_{-1, -1}(  \lambda_1, \lambda_2)$.
 
 The proof follows.
 \end{proof}

 The proof of the following result is analogous to that of Proposition \ref{step-1}. Alternatively it follows directly from Proposition \ref{step-1} together with spectral flow, which we will discuss in the next section, see in particular \eqref{SF}.
\begin{proposition}   \label{step-2-2}
For each $ w \in \Pi_{-1, -1}(  \lambda_1, \lambda_2)$ there is $y \in   U(  \widehat{\mathfrak b})$ such that
$ y. w \in   \Pi_{-1, -1}(  \lambda_1, \lambda_2)_{top}. $ 
\end{proposition}

\subsection{ $\mathcal S \otimes \Pi(0)$--modules as $\widehat{\mathfrak b}$--modules}

Consider the following  $\mathcal S \otimes \Pi(0)$--module $$\mathcal S_{r_2} (\lambda_2):= 
\mathcal S \otimes \Pi(0). e^{ r_2 d_2 /2 + \lambda_2 c_2}. $$
\begin{proposition} \label{step-3-weyl}
Assume that $w \in \mathcal S_{r_2} (\lambda_2)$ and $r_2 \in \{-1, 0\}$. Then there exist 
  $y \in U(  \widehat{\mathfrak b} )$  and $\ell \in {\Z}$ such that  $$y . w = e^{ r_2 d_2 / 2 +    (\ell + {\lambda_2}) c_2}. $$
\end{proposition}
\begin{proof}
By using Propositions \ref{step-1} and  \ref{step-2-2} we construct  $y_1 \in  U(  \widehat{\mathfrak b} )$ such that
$y_1 w = e^{ r_2 d_2 / 2 +  \ell_1  c_1 +  \bar \ell_2 c_2}$ for certain $\ell_1, \ell_2 \in {\Z}$.  Since $y_1 w \in  \mathcal S_{r_2} (\lambda_2)$, we conclude that $e^{\ell_1 c_1} \in \mathcal S$, which implies that $\ell_1 \ge 0$, and therefore $e^{\ell_1 c} =
: \beta  ^{\ell_1}:$. By applying the action of $e_1 = -\gamma e^{c_2}$, we get
\bea
e_1(-1) ^{\ell_1} y_1 w &=& \nu_1  e^{ r_2 d_2 / 2  + (   \bar \ell_2 +\ell_1)  c_2}  \quad \mbox{if} \ r_2 =0,  \nonumber \\
e_1(0) ^{\ell_1} y_1 w &=& \nu_2  e^{ r_2 d_2 / 2  + (   \bar \ell_2 +\ell_1)  c_2}  \quad \mbox{if} \ r_2 =-1,  \nonumber 
\eea
where $\nu_1, \nu_2 \ne 0$. The proof follows.

\end{proof}
\section{Spectral flow}\label{spec-flow}

Let $V$ be a vertex algebra and let 
\[
h(z) = \sum_{n \in \mathbb Z} h(n) z^{-n-1} 
\]
be the generator of a Heisenberg vertex subalgebra of $V$. Define the Li $\Delta$-operator \cite{LiPhy97}\footnote{Note that our convention differs by a sign from the one used in \cite{LiPhy97}}
\[
\Delta(h, z) := z^{-h(0)} \prod_{n=1}^\infty \text{exp}\left( \frac{(-1)^n { h(n)}}{n}\right)
 \]
and denote by
\[
\gamma^h(A(z)) = Y(\Delta(h, z)A, z)
\]
 the field $A(z)$ twisted by $\Delta(h, z)$. 
In particular for a $V$-module $M$ one defines the spectrally flow  twisted module $\gamma^{-h}(M)$ to be $M$ with the action of $V$ given by $\gamma^h(A(z))$ for $A\in V$. 
The following two special cases are important
\begin{enumerate}
\item If $h(z)A(w) = \alpha (z-w)^{-2}$, then $\gamma^h(A(z)) = A(z) - \alpha z^{-1}$
\item If $h(z)A(w) = \alpha A(w) (z-w)^{-1}$, then $\gamma^h(A(z)) = z^{-\alpha}A(z)$
\end{enumerate}
The first one is obtained in the proof of Proposition 3.4 of \cite{LiPhy97} and the second one just means that if $h(0)$ acts by multiplication with $\alpha$ then $\Delta(h, z)$ acts by $z^{-\alpha}$.
In particular if $V$ is a lattice vertex algebra $V_L$ and $h(0)=\mu$ in $L'$ then $\gamma^{-h}(V_{\lambda+L})= V_{\lambda -\mu+L}$ for any $\lambda \in L'$. 
This follows since the first case tells us that the action of the zero-modes of the Heisenberg subalgebra is shifted by $-\beta$ and it has been first proven in Proposition 3.4 of  \cite{LiPhy97}.
Shifting the zero-mode of the Heisenberg vertex algebra is an automorphism of the Heisenberg vertex algebra. In general, for a $V$-module $M$ and an automorphism $\gamma$, we denote by $\gamma(M)$ the $V$-module whose underlying vector space is isomorphic to $M$, the isomorphism is denoted by $\gamma$ as well and the action of $V$ is given by 
\[
 A \gamma(w) = \gamma\left(\gamma^{-1}(A)w\right), \qquad w\in M. 
\]
We consider now $ \BP \otimes \Pi(0)^{\otimes 2}$.
Recall that we set 
\[
\mu_1 = \frac{1}{2}c_1-\frac{1}{2}d_1-\frac{2k+9}{6}c_2+\frac{1}{2}d_2, \qquad \mu_2 = -c_1+d_1+\frac{4k+9}{6}c_2+\frac{1}{2}d_2.
\]
Define $\lambda^{a, b}$ as the spectral flow corresponding to $a\mu_1 + b\mu_2$ for integers $a, b$. From above discussion we obtain
\begin{equation}
\begin{split}
\lambda^{a, b}(c_1(z)) &= c_1(z) - (2b-a)z^{-1}, \\
\lambda^{a, b}(d_1(z)) &= d_1(z) + (2b-a)z^{-1}, \\
\lambda^{a, b}(c_2(z)) &= c_2(z) - (a+b)z^{-1}, \\
\lambda^{a, b}(d_2(z)) &= d_2(z) + \left(\frac{2k+9}{6}a -\frac{4k+9}{6}b \right)z^{-1}.
\end{split}
\end{equation}
From this we see that the action of $c_1(0)$ is shifted by $2b-a$, the one of $d_1(0)$ by $-(2b-a)$ and similar for $c_2(0), d_2(0)$, in particular we can identify
\begin{equation}\label{SF}
\lambda^{a, b}\left(\Pi_{r_1, r_2}(  \lambda_1, \lambda_2)\right) \cong \Pi_{r_1+2b-a, r_2+a+b}(  \lambda_1 +\frac{a}{2}, \lambda_2 + \frac{k}{6}(2b-a)+\frac{1}{4}(a-b)).
\end{equation}

as Virasoro field we take 
$$L^ {\Pi}  =  \frac{k}{3}\partial c_2 -\frac{1}{2} \partial d_1 -\frac{1}{2} \partial d_2 + \frac{1}{2}:(c_1 d_1 + c_2 d_2): $$ and its image under spectral flow is obtainend from the spectral flow action on the Heisenberg fields. 
\begin{equation}
\begin{split}
\lambda^{a, b}(L^\Pi) &= L^\Pi - z^{-1}\mu_{a, b} +  
 z^{-2}\left(-\frac{(2b-a)^2}{2} -(a+b)\left(\frac{2k+9}{6}a -\frac{4k+9}{6}b \right)\right. \\
&\qquad  \left. +\frac{k}{3}(a+b)+\frac{2b-a}{2} +\left(\frac{2k+9}{6}a -\frac{4k+9}{6}b \right)\right)  \\
= L^\Pi - &z^{-1}\mu_{a, b} + 
   z^{-2}\left(k (a^2+b^2- ab)  - \frac{(2k+3)(b-2a)(b-2a+1)}{6}  \right)
\end{split}
\end{equation}
Denote by $\sigma^\ell$ the spectral flow corresponding to $\ell J$. From section 2 of \cite{AKR-2020} we obtain that 
\begin{equation}
\begin{split}
\sigma^\ell(G^\pm(z)) &= z^{\mp \ell} G^{\pm}(z)\\
\sigma^\ell(J(z)) &= J -\frac{2k+3}{3}\ell z^{-1} \\
\sigma^\ell(L^{\text{BP}}(z)+\frac{1}{2}\partial J(z)) &= L^{\text{BP}}(z)+\frac{1}{2}\partial J(z) -\ell z^{-1}J(z) +\frac{2k+3}{3}\frac{\ell(\ell+1)}{2}z^{-2} 
\end{split}
\end{equation}
Note that their Virasoro field is $ \widetilde L = L^{\text{BP}}(z)+\frac{1}{2}\partial J(z)$.
Let $\gamma^{a, b}$ be the spectral flow corresponding to $ah_1 + bh_2$. Then via the embedding in $ \BP \otimes \Pi(0)^{\otimes 2}$ we have
\[
\gamma^{a, b} =  \sigma^{b-2a} \otimes \lambda^{a, b},
\]
since $h_1 = -2J + \mu_1$ and $h_2 =J + \mu_2$. 
It acts on generators and Virasoro field $L = L^{\text{BP}}(z)+\frac{1}{2}\partial J(z) + L^\Pi$ as follows
\begin{equation}
\begin{split}
\gamma^{a, b}(e_1(z)) &= z^{-2a+b}e_1(z) \\  
\gamma^{a, b}(e_2(z)) &= z^{a-2b}e_2(z) \\  
\gamma^{a, b}(e_3(z)) &= z^{-a-b}e_3(z) \\  
\gamma^{a, b}(h_1(z)) &= h_1(z) -k(2a-b)z^{-1}\\
\gamma^{a, b}(h_2(z)) &= h_2(z) -k(-a+2b)z^{-1}\\
\gamma^{a, b}(f_1(z)) &= z^{2a-b}f_1(z) \\  
\gamma^{a, b}(f_2(z)) &= z^{-a+2b}f_2(z) \\  
\gamma^{a, b}(f_3(z)) &= z^{a+b}f_3(z) \\
\gamma^{a, b}(L(z)) &= L(z) -z^{-1}(ah_1(z) +bh_2(z))   +z^{-2}k(a^2+ab+b^2).  
\end{split}
\end{equation}
Let $M$ be a module for $ \BP \otimes \Pi(0)^{\otimes 2}$ and define its character as
\[
\text{ch}[M](q, z_1, z_2) := \text{tr}_M\left( q^{L(0)- \frac{c}{24}} z_1^{h_1(0)}z_2^{h_2(0)} \right). 
\]
Then  
\[
\text{ch}[\gamma^{a, b}(M)](q, z_1, z_2) =  q^{k(a^2+ab+b^2)} z_1^{k(2a-b)} z_2^{k(2b-a)}\text{ch}[M](q, z_1q^a, z_2q^b). 
\]
  
\section{ Realization of the simple affine VOA $L_k(\mathfrak{sl}_3)$  }\label{rel-simpleVOA}

Using Proposition \ref{step-3-weyl} and the action of  $U(  \widehat{\mathfrak b} )$ we get immediately:
\begin{lemma} \label{step-2}
Assume that $U$ is any $V^k(\g)$--submodule of $\mathcal W_k \otimes \mathcal S \otimes  \Pi(0)$. 
Then $U$ contains vector $ Z \otimes e^{   \ell_2 c_2} $ where $ Z \in  \mathcal W_k$ and $  \ell_2 \in \mathbb Z$.
\end{lemma}

  Let $\mathcal W_k$ be the  simple Bershadsky-Polyakov algebra at level $k$.  We choose the Virasoro field $\widetilde  L =  L^\mathrm{BP} +\frac{1}{2}\partial J $, so that. $G^+$ (resp. $G^-$) is a primary vector of weight $1$ (resp. $2$).  Then $\mathcal W_k$  is generated by fields:
 \bea && J(z) = \sum _{n = 0} ^{\infty} J(n) z^{-n-1}, \ \widetilde L   (z) = \sum_{n=0} ^{\infty} \widetilde L(n) z^{-n-2}, \nonumber \\
 && G^{+} (z) = \sum_{n=0}  ^{\infty} G^+ (n) z^{-n-1}, \ G^{-}(z) = \sum_{n=0} ^{\infty}  G^{-} (n) z^{-n-2}\nonumber \eea
 
 Let $$\Phi: V^k (\mathfrak{sl}_3) \rightarrow \mathcal W_k \otimes \mathcal S(1) \otimes \Pi(0) \subset \mathcal W_k \otimes \Pi(0) ^{\otimes 2}$$  as  in  Section \ref{realization-1}. Let $\widetilde L_k(\mathfrak{sl}_3) = \operatorname{Im}(\Phi)$.

\begin{theorem} \label{main} We have: 
$\widetilde L_k(\mathfrak{sl}_3)$ is simple if and only if $k \notin {\Z}_{\ge 0}$. In particular, for $k \notin {\Z}_{\ge 0}$, there exist an embedding $L_k(\mathfrak{sl}_3) \hookrightarrow \mathcal W_k \otimes \mathcal S(1) \otimes \Pi(0)$.
\end{theorem}



\begin{corollary}
If   $R$ is a $\mathcal W_k$--module and $k \notin {\Z}_{\ge 0}$, then   $R \otimes   \Pi_{-1, -1} (\lambda_1, \lambda_2)$  is an $L_k(\mathfrak{sl}_3)$--module.
\end{corollary}

\begin{remark} Note the following:
\begin{itemize}
\item We can choose $R$ so that $R_{top}$ is finite-dimensional. Then all weight spaces for $\mathfrak{sl}_3$--module  $R_{top} \otimes \Pi_{-1,-1} (\lambda_1, \lambda_2) _{top}$  are finite-dimensional too.
In particular, this happens for the collapsing level $k=-3/2$ and the levels $k=-3/2 + n$ and $n$ in $\mathbb Z_{\geq 0}$ where $\mathcal W_k$ is strongly rational \cite{Ar}.
\item For $k=-9/4$ ($B_p$ algebra for $p=4$ case \cite{CRW}),  when $R_{top}$ is finite-dimensional, one shows that then $R_{top}$ is $1$-dimensional \cite{AK}.
\item But for $k=-9/4$ (and more generally for $k$ non-degenerate principal admissible) we have irreducible modules $R$ such that $R_{top}$ is infinite-dimensional. Then  all weight spaces of $\mathfrak{sl}_3$--module $R_{top} \otimes \Pi_{-1,-1} (\lambda_1, \lambda_2) _{top}$  will be infinite-dimensional. So we can not apply character argument. Our proof covers this case too.
\end{itemize}

\end{remark}

  \begin{proof}
  Assume that $\widetilde L_k(\mathfrak{sl}_3)$ is not simple. 
   Then it contains a non-trivial ideal 
  $\mathcal I $, which,  by Lemma \ref{step-2},  contains vector
  $$ Z \otimes e^{  \ell_2 c_2}  \in \mathcal I , \quad (Z \in \mathcal W_k,   \ell_2 \in \mathbb Z). $$
  We claim that 
 \bea \label{vacuum} {\bf 1}  \otimes  e^{  \ell_2 c_2} \in \mathcal  I. \eea
  
  In the proof of (\ref{vacuum}) we use  the simplicity of $\mathcal W_k$ and the  action of elements $f_j (k)$, $h_1 (k)$, $k \ge 1$. 
   Let us prove this claim. First we notice that since $\mathcal W_k$ is simple, then   either $Z \in  {\C} {\bf 1} $  or $Z$ can not be a singular vector for $\mathcal W_k$. 

Assume that $Z \notin  {\C} {\bf 1} $. Then one of the following holds for certain $n \in {\Z}_{\ge 1}$:
\begin{itemize}
\item[(a)] $J(n) Z \ne 0$,
\item[(b)]
 $G^+ (n) Z \ne 0$, 
\item[(c)] 
 $G^- (n)  Z \ne 0$, 
 \item[(d)] $L(n)^\mathrm{BP} Z \ne 0$.
 \item[] (Note that the case (d) doesn't need to be considered at  the critical level $k=-3$ since  then $L(n)^\mathrm{BP} $ acts as a scalar on $\SBP \otimes \mathcal S \otimes \Pi(0)$.)
\end{itemize}
 Let us analyze these cases

     \begin{itemize}
  \item[(a)] If $J(n) Z \ne 0$. 
  Then since $e^{  \ell_2 c_2}$ is a highest-weight vector for $c_1, c_2, d_1$ and $d_2$ 
  and since $$h_1 = -2J+\frac{1}{2}c_1-\frac{1}{2}d_1-\frac{2k+9}{6}c_2+\frac{1}{2}d_2$$ and
$$h_2 = J-c_1+d_1+\frac{4k+9}{6}c_2+\frac{1}{2}d_2$$  
  we have
  \begin{equation}\nonumber
  \begin{split}
    h_i  (n). (Z \otimes e^{  \ell_2 c_2}) &= \nu_i J(n)  Z   \otimes e^ { \ell_2 c_2} \\
     &= Z_1 \otimes e^{\ell_1 c_1 + \ell_2 c_2}  \quad 
     \end{split} 
     \end{equation}
  for certain $Z_1$, $\mbox{wt} (Z_1) < \mbox{wt}(Z)$ and $\nu_1 =-2, \nu_2 = 1$.
  \item[]
  
  \item[(b)] If $G^+  (n)  Z  \ne 0$ and $J(i)  Z = 0$ for all $i \ge 1$. Then since $$f_1 = G^+ - :\left( 2J -(k+1)c_1+\frac{8k+9}{6}c_2-\frac{1}{2}d_2 \right)\mathrm{e}^{c_1-c_2}:$$ we have 
 \begin{equation}\nonumber
  \begin{split}
   f_1 (n). (Z \otimes e^{  \ell_2 c_2}) &=  G^+  (n)   Z   \otimes e^{  \ell_2 c_2}  = Z_2 \otimes e^{  \ell_2 c_2}  
   \end{split} 
     \end{equation}
  for certain $Z_2$, $\mbox{wt} (Z_2) < \mbox{wt}(Z)$.
  
  \item[]
  
 \item[(c)] Let $G^-  (n)  Z  \ne 0$ and $J(i) Z = G^- (n+i) Z =  0$ for all $i \ge 1$. Then  since 
 \begin{equation}\nonumber
  \begin{split}
 f_2 &= G^- \mathrm{e}^{-c_2}-\frac{k+1}{2}:(\partial c_1 + \partial d_1)\mathrm{e}^{-c_1}:  \\
&\quad - \frac{1}{2}:\left( J -\frac{2k+3}{2}c_1+\frac{1}{2}d_1+\frac{4k+9}{6}c_2+\frac{1}{2}d_2  \right)(c_1+d_1)\mathrm{e}^{-c_1}:
 \end{split} 
     \end{equation}
we have 
  \begin{equation}\nonumber
  \begin{split}
   f_2  (n ). (Z \otimes e^{  \ell_2 c_2}) &= G^-  (n) Z   \otimes e^{-c_2} _{-1} e^{ \ell_2 c_2} \\ &= Z_3 \otimes e^{ ( \ell_2 -1)c_2}    \end{split} 
     \end{equation}
  for certain $Z_3$, $\mbox{wt} (Z_3) < \mbox{wt}(Z)$.
  
  \item[]
  \item[(d)]   Let $G^{\pm }  (i)  Z =    J (i)  Z = L(n+i) Z =  0$ for all $i >0$ and $L(n) Z \ne 0$. Then since
  \begin{equation}\nonumber
  \begin{split}
  f_3 &=  -\frac{1}{2} :G^+(c_1+d_1)\mathrm{e}^{-c_1}:-G^-\mathrm{e}^{c_1-2c_2}\\
&\qquad+ :\biggl( (k+3)L^\mathrm{BP} + \frac{k+1}{2}\partial \left( J+c_1+\frac{2}{3}k c_2-d_2 \right)\biggr)\mathrm{e}^{-c_2}:\\
&\qquad +:\biggl\{ -J \left( J+c_1-d_1+\frac{2k-9}{6}c_2-\frac{1}{2}d_2 \right) \\
&\qquad-\frac{1}{12}(c_1-d_1) \left( (8k+9)c_2 - 3 d_2 \right) -\frac{k+2}{2}c_1 d_1\\
&\qquad -\frac{4k^2-18k-27}{36} (c_2)^2 +\frac{k}{3}c_2 d_2 -\frac{1}{4}(d_2)^2 \biggr\}\mathrm{e}^{-c_2}:.
   \end{split} 
     \end{equation} we have
     
 \begin{equation}\nonumber
  \begin{split}
   f_3  (n +1). (Z \otimes e^{  \ell_2 c_2}) &=  (k+3) L^\mathrm{BP}(n)  Z   \otimes e^{-c_2} _{-1} e^{  \ell_2 c_2}  \\ &= Z_4 \otimes e^{ ( \ell_2 -1)c_2}    
   \end{split} 
     \end{equation}
 for certain $Z_4$, $\mbox{wt} (Z_4) < \mbox{wt}(Z)$.
 
  \end{itemize}

  By applying (a)-(d) and simplicity of $\mathcal W_k$ we conclude that ${\bf 1} \otimes e^{  \ell_2 c_2} \in \mathcal I$. 
 Using the operator  $ e_3(-1)$ we  can conclude that
 $ {\bf 1}  \otimes  e^{  \ell_2 c_2} \in \mathcal  I$ for certain   $\ell_2 \in {\Z}_{\ge 0}$.
This implies that $  e_3 (-1) ^{\ell _2}  {\bf 1}$ belongs to the  maximal submodule of $V^k (\mathfrak{sl}_3)$. But this can only happen  if $k \in {\Z}_{\ge 0}$.
 \end{proof}

  \begin{remark}
Note that the proof of the previous theorem also works at the critical level $k=-3$. For $\SBP$ we can take any simple quotient of $\BP$. Only the case (d)  doesn't need to be considered. 

Affine vertex algebras at the critical level are of course interesting in its own right, e.g. since they have a large center. In addition, vertex algebras that are extensions of three copies of the affine vertex algebra at the critical level play the crucial role in the Moore-Tachikawa conjecture on vertex algebras of class $S$ \cite{MT}, proven by Arakawa \cite{Ar-MT}. These extensions are such that the center is identified. In the case $\mathfrak{g} = \mathfrak{sl}_2$ this algebra is just four $\beta\gamma$ vertex algebras and in fact in general $n^2$ $\beta\gamma$ vertex algebras are an extension of two affine vertex algebras of type $\mathfrak{sl}_n$ at the critical level times a subregular $\mathcal W$-algebra of $\mathfrak{sl}_n$ at the critical level \cite{CGL}. For $n=3$ we expect that our procedure lifts nine $\beta\gamma$ vertex algebras to $L_{-3}(\mathfrak{e}_6)$.
 \end{remark}
 
  \section{Singular vectors and realizations of $V^k(\mathfrak{sl}_3)$--modules in $KL_k$}
  \label{sect-KL}
  
  In this section we shall give a realisation of all irreducible modules in $KL_k$ for $\mathfrak{sl}_3$ inside the $\BP\otimes \mathcal S \otimes \Pi(0)$--modules of type
  $L[x,y] \otimes S_0(\lambda_2) $, where $L[x,y]$ is an irreducible, highest weight $\BP$--module such that $L[x,y]_{top}$ is finite-dimensional. For that purpose we shall first  study singular vectors in  $L[x,y] \otimes S_0(\lambda_2) $.

  The vector $w \in \Pi_{0,0}(  \lambda_1, \lambda_2)$ is called singular for $\widehat{\mathfrak b}$ if
  $$  \bar h(n+1) w = e_i (n) w = 0 \quad \mbox{for}  \ i=1,2,3, n \ge 0. $$
  Again Proposition \ref{step-3-weyl} gives us immediately
  \begin{lemma} \label{sing-pi00}
  If $u$ is a singular vector in $\Pi_{0,0}(  \lambda_1, \lambda_2)$, then $\lambda_1 \in {\Z}$ and $u= Ae^{ (m+ \lambda_2) c_2} $ for $ A\ne 0$ and $m \in {\Z}$. 
  \end{lemma}
 Assume that $Z$ is any weight $\mathcal W^k$--module.  We say that a vector $v\in Z$ is a singular vector of weight $(x,y)$ if 
   $J(0) v = x v$,  $\widetilde{L}(0) v = y v$, and 
 $$G^+ (n+1) v = G^-(n) v = \widetilde{L}(n+1)  v  = J(n+1) v = 0, \quad (n \in {\Z}_{\ge 0}). $$

  \begin{proposition} Assume that $Z$ is any weight $\mathcal W^k$--module. Then
   $ w \in Z \otimes \Pi_{0,0}(  0, \lambda_2)$ is a singular vector for $\widehat{\mathfrak{sl}}_3$ if and only if $w$ has the form
 \bea  && w= v \otimes e^{\bar m c_2} \label{singular-tensor} \eea
   where $v$ is a singular vector for $\mathcal W^k$ of highest weight $(x,y)$, $\bar m = m + {\lambda_2}$  and  \bea { (k+3)  (y + \tfrac{x}{2})  -\tfrac{k+1}{2}   (x-  2 \bar m )   - x ( x- \bar m) -   \bar m ^2 =0.} \ \label{eq-sing-1} \eea
  \end{proposition}
  \begin{proof}
  By using Lemma \ref{sing-pi00} and the same arguments as in the proof of Theorem \ref{main} we get that any singular vector  for $\widehat{\mathfrak{sl}}_3$ must be of the form (\ref{singular-tensor}). Let us now find a sufficient condition that $v \otimes e^{\bar m c_2}$ is a singular vector. Assume that $v$ is a highest weight vector of the highest weight $(x,y)$.
 Since we already know that $w$ is a singular for   $\widehat{\mathfrak b}$, it remains to check conditions $f_i(1) w = 0$ for $i=1,2,3$.
 
 One easily see that $f_1(1) w= f_2 (1) w=  0$  without any further condition.
 
By direct calculation we get 
 \bea
 f_3(1) w =  \left( (k+3)  (y + \tfrac{x}{2})  -\tfrac{k+1}{2} (x-  2 \bar m )   - x ( x- \bar m) -   \bar m ^2\right)  v\otimes e^{(\bar m-1)c_2}. \nonumber 
 \eea
 This implies that $w$ is a singular vector if and only if  $$(k+3)(y + \tfrac{x}{2})  -\tfrac{k+1}{2}  (x-  2 \bar m )   - x ( x- \bar m) -   \bar m ^2=0. $$ 
   \end{proof}

  Let $a,b \in {\C}$. Consider the highest weight $\mathcal W^k$--module 
$
  L[x,y]$ with highest weight vector $v[x,y]$ and highest weight
  \bea
  x&=& \frac{b-a}{3},   \label{def-x} \\
  y&=& \frac{ (b - a) ^2 - 3 (a + b) (2 (k + 1) - a - b) } {12 (k + 3)}
      - \frac{b-a}{6}. \label{def-y}  
  \eea
Then the equation  (\ref{eq-sing-1}) has two solutions:
$$ m_1:=\frac{ a + 2b}{3}, \quad m_2:= \frac{3 - 2a -  b}{3} + k. $$ Then
$$v[x,y] \otimes e^{m_1 c_2}, \quad v[x,y] \otimes e^{m_2 c_2} $$
are singular vectors for $V^k(\mathfrak{sl}_3)$. 
  %

  Define the following simple current extension of $\Pi(0)$:
  \bea \label{pi-ext} \Pi(0) ^{1/3} = \Pi(0) \oplus \Pi(0) . e^{\frac{c_2}{3}} \oplus \Pi(0). e^{ \frac{2 c_2}{3}}. \eea
 \begin{theorem} Assume that $a,b \in {\Z}_{\ge 0}$, $k \notin {\Z}_{\ge 0}$. Then we have:
 \begin{itemize}
 \item[(1)]
 $ L^{\mathfrak{sl}_3} _k(a \omega_1 + b \omega_2) = V^k(\mathfrak{sl}_3). \left(v[x,y] \otimes  e^{ m_1 c_2 } \right) \subset L[x,y] \otimes \mathcal S_0(m_1)$.
    \item[(2)] $ L^{\mathfrak{sl}_3} _k(a \omega_1 + b \omega_2) = \left( L[x,y] \otimes \mathcal S \otimes \Pi(0) ^{1/3} \right) ^{\mbox{int}_{\mathfrak{sl}_3}}. $
   \end{itemize}
 \end{theorem}
  \begin{proof}
  Using realization we get that the $\mathfrak{sl}_3$--highest weights of
  \begin{itemize}
  \item  $v[x,y] \otimes  e^{m_1 c_2 } $ is $a \omega_1 + b\omega_2$,
  \item   $v[x,y] \otimes  e^{m_2 c_2 } $ is $(k+1-b) \omega_1 + (k+1-a) \omega_2$.
  \end{itemize}
  
  Let $\widetilde L^{\mathfrak{sl}_3} _k(a \omega_1 + b \omega_2) = V^k(\mathfrak{sl}_3). \left(v[x,y] \otimes  e^{m_1c_2 } \right)$.
  
 Since there exist at most  two linearly independent singular vectors in
 $L[x,y] \otimes \mathcal S_0(m_1)$,  and the $\mathfrak{sl}_3$--weight 
 of  the singular vector $v[x,y] \otimes  e^{m_2 c_2 } $ is not dominant integral, we conclude:
\bea 
  &&   \widetilde L^{\mathfrak{sl}_3} _k(a \omega_1 + b \omega_2) \quad \mbox {is irreducible} \nonumber  \\ \iff && 
    v[x,y] \otimes  e^{m_2 c_2} \notin  \widetilde  L^{\mathfrak{sl}_3} _k(a \omega_1 + b \omega_2) \nonumber \\ \iff && 
     \widetilde L^{\mathfrak{sl}_3} _k(a \omega_1 + b \omega_2) _{top} =   L^{\mathfrak{sl}_3} _k(a \omega_1 + b \omega_2)_{top} \nonumber \eea

  Next we recall that 
  $$ L[x,y]_{top} = i \iff  G^+(0) ^i v[x,y] = 0 \iff h_i(x,y) =0. $$
  In our case we get $h_i(x,y) = (1+a-i) (b+i-k-2)$, which implies that $h_{a+1} (x,y) =0$. So $ \dim L[x,y]_{top} = a+1  < \infty$.
  
  Using the automorphism $\rho$ of $\mathcal W^k$ such that $G^+\mapsto G^-$ we get that in our case $\rho(  L[x,y] ) = L[-x, y+x]$.
  Since $h_j(-x, y+x) = (1+b-i) (a+i-k-2)$, we conclude that $\dim L[-x, y+x]_{top}  = b+1 < \infty$, implying that in  $L[x,y]$
  $$ G^-(-1) ^{b+1} v[x,y] =0. $$

  We get
  $$ f_1(0) ^{a+1}  ( v[x,y] \otimes  e^{m_1c_2 } ) = G^+(0)^{a+1} v[x,y] \otimes  e^{m_1c_2 }  = 0. $$
   $$ f_2(0) ^{b+1} ( v[x,y] \otimes  e^{m_1c_2 } ) = G^-(-1)^{b+1} v[x,y] \otimes  e^{m_1 c_2 - (b+1) c_2 }  = 0. $$
   
   So $\widetilde L^{\mathfrak{sl}_3} _k(a \omega_1 + b \omega_2) _{top} \cong    L^{\mathfrak{sl}_3} _k(a \omega_1 + b \omega_2)_{top}$, and therefore $\widetilde L^{\mathfrak{sl}_3} _k(a \omega_1 + b \omega_2)$  is irreducible. This  proves assertion (1).  The proof of assertion (2) follows from (1) and the fact that $v\otimes e^{\frac{a+2b}{3}}$ is, up to a scalar factor, unique singular vector in  $L[x,y] \otimes \mathcal S \otimes \Pi(0) ^{1/3}$ whose weight is dominant, integral with respect to $\mathfrak{sl}_3$. \end{proof}

   We also note the following important consequence of the previous theorem.
   
   \begin{corollary}
   For $a, b \in {\Z}_{\ge 0}$ and $k \notin {\Z}_{\ge 0}$  we have:
   \begin{itemize}
   \item $H_{f_{min}}\left(L^{\mathfrak{sl}_3} _k(a \omega_1 + b \omega_2)\right) = L[x,y]$, where  the weight $(x,y)$ is defined by (\ref{def-x})-(\ref{def-y}).
   \item  $ L^{\mathfrak{sl}_3} _k(a \omega_1 + b \omega_2)$ is a $L_k(\mathfrak{sl}_3)$--module if and only if $L[x,y]$ is a $\mathcal W_k$--module.
   \end{itemize}
   \end{corollary}

  \section{Relaxed $L_k(\mathfrak{sl}_3)$--modules}
  \label{relaxed-constr-general}

\begin{theorem}  \label{almost-sl3} Assume that $R$  is an irreducible ${\Z}_{\ge 0}$--graded   $\mathcal W_k$--module   with top component $R_{top}$. 
Then  $R \otimes   \Pi_{-1, -1} (\lambda_1, \lambda_2)$  is an almost irreducible $L_k(\mathfrak{sl}_3)$--module in the sense that any submodule $W$ of $L \otimes   \Pi_{-1, -1} (\lambda_1, \lambda_2)$ intersects $R_{top} \otimes   \Pi_{-1, -1} (\lambda_1, \lambda_2)_{top}$ non-trivially.
\end{theorem}
\begin{proof}
Assume that $W$ is any submodule of  $R \otimes   \Pi_{-1,-1} (\lambda_1, \lambda_2)$, take any weight vector $w \in  W$. Using  the action of the Borel subalgebra $\widehat{\frak b}$ and  Lemma  \ref{ired-relaxed-1}  we get a non-zero  weight vector $$z = u \otimes e^{- \frac{1}{2} d_1 - \frac{1}{2} d_1 + t_1 c_1 + t_2 c_2 } \in W, $$  where $u$ is element of $R$, $t _1 -\lambda_1, t_2 - \lambda_2  \in {\Z}$. 
 
 Now we proceed as in the proof of Theorem \ref{main}.

 \begin{itemize} 
 \item  Assume that $u \notin  R_{top}$. Then one of the following holds for certain $n \in {\Z}_{\ge 1}$:
  $$G^+ (n)  u  \ne 0, \ G^- (n)  u \ne 0, \ L(n) u  \ne 0, \ J(n) u  \ne 0. $$
  \item Let us analyze these cases
  \item[]
  \item[(a)] Let $J(n) u  \ne 0$. Then $(h_i, - \frac{1}{2} d_1 - \frac{1}{2} d_1 + t_1 c_1 + t_2 c_2  ) \ne 0 $ for certain $i\in \{ 1, 2\}$. This implies  that 
  \bea   h_i  (n). (u  \otimes  e^{- \frac{1}{2} d_1 - \frac{1}{2} d_1 + t_1 c_1 + t_2 c_2 } )  &= & \nu J(n) u   \otimes e^{- \frac{1}{2} d_1 - \frac{1}{2} d_1 + t_1 c_1 + t_2 c_2 }  \nonumber \\
   &=& u_1 \otimes e^{- \frac{1}{2} d_1 - \frac{1}{2} d_1 + t_1 c_1 + t_2 c_2 }  \quad (\nu \ne 0) \nonumber \eea
  for certain $u_1$, $\mbox{wt} (u_1) < \mbox{wt}(u)$.
  \item[]
  
  \item[(b)] Let $G^+ (n)  u  \ne 0$ and $J(i ) u = 0$ for all $i \ge 0$. Then 
 \bea  f_1 (k). (u \otimes  e ^{- \frac{1}{2} d_1 - \frac{1}{2} d_1 + t_1 c_1 + t_2 c_2 } ) &=&  (G^+ (n)   + J(n) e^{c_1-c_2}_{-1} )u   \otimes e^{- \frac{1}{2} d_1 - \frac{1}{2} d_1 + t_1 c_1 + t_2 c_2 }  \nonumber \\   &=& u_2 \otimes e^{- \frac{1}{2} d_1 - \frac{1}{2} d_1 + t_1 c_1 + t_2 c_2 }   \nonumber \eea
  for certain $u_2$, $\mbox{wt} (u_2) < \mbox{wt}(u)$.
  
  \item[]
  
 \item[(c)] Let $G^- (n)  u  \ne 0$ and $J(i)  u =  G^- (n+i) u = 0$ for all $i >0$. Then 
 \bea  f_2  (k+1). (u \otimes  e^{- \frac{1}{2} d_1 - \frac{1}{2} d_1 + t_1 c_1 + t_2 c_2 } ) &=&  G^-  (k) u   \otimes e^{-c_2} _{-2} e^{- \frac{1}{2} d_1 - \frac{1}{2} d_1 + t_1 c_1 + t_2 c_2 } \nonumber \\ & =& u_3 \otimes  e^{- \frac{1}{2} d_1 - \frac{1}{2} d_1 + t_1 c_1 +( t_2 -1) c_2 }  \nonumber \eea
  for certain $u_3$, $\mbox{wt} (u_3) < \mbox{wt}(u)$.
  
  \item[]
  \item[(d)]   Let $G^{\pm } (i)   u =    J(i)   u =  L(n+i) u = 0$ for all $i >0$ and $L(n) u \ne 0$. Then
  \bea  f_3  (n). (u \otimes e^{- \frac{1}{2} d_1 - \frac{1}{2} d_1 + t_1 c_1 + t_2  c_2 } ) &=&  (k+3) L(n)  u    \otimes e^{-c_2} _{-2} e^{- \frac{1}{2} d_1 - \frac{1}{2} d_1 + t_1 c_1 +( t_2 -1) c_2 }  \nonumber \\  &=&  u_4 \otimes  e^{- \frac{1}{2} d_1 - \frac{1}{2} d_1 + t_1 c_1 +( t_2 -1) c_2 }   \nonumber \eea
 for certain $u_4$, $\mbox{wt} (u_4) < \mbox{wt}(u)$.
 
 \item[]
 \item By applying (a)-(d) and simplicity of the $\mathcal W_k$--module $R$ we conclude that $\bar u  \otimes e^{- \frac{1}{2} d_1 - \frac{1}{2} d_1 + t_1 c_1 +( t_2 -1) c_2 }  \in W$ for certain $\bar u \in R_{top}$.

 \end{itemize}
The proof follows.
\end{proof}
\begin{remark}
The proof of the Theorem  \ref{almost-sl3} also hold for $k=-3$.  The main difference with the non-critical case  is that as  in the proof of Theorem  \ref{main} the case (d) doesn't need to be considered.
\end{remark}
\subsection{Irreducibility of relaxed modules}

\begin{theorem} \label{thm:irred}  Assume that $R$  is an irreducible ${\Z}_{\ge 0}$--graded   $\mathcal W_k$--module   with top component $R_{top}$.  Assume that $R_{top} \otimes \Pi_{-1,-1} (\lambda_1, \lambda_2) _{top}$ is irreducible $\mathfrak{sl}_3$--module.
Then  $R \otimes   \Pi_{-1, -1} (\lambda_1, \lambda_2)$  is an   irreducible $L_k(\mathfrak{sl}_3)$--module.
\end{theorem}
\begin{proof}
Assume first that $k \ne -3$. We need to prove that every vector $w \in R \otimes   \Pi_{-1, -1} (\lambda_1, \lambda_2)$  is cyclic.  Let $W = L_k(\mathfrak{sl}_3). w$. By Theorem \ref{almost-sl3}, we have that  there is a vector $0\ne w_1 \in W \cap R_{top} \otimes \Pi_{-1,-1} (\lambda_1, \lambda_2) _{top}$. By assumption, $U(\mathfrak{sl}_3). w_1 = L_{top} \otimes \Pi_{-1,-1} (\lambda_1, \lambda_2) _{top}$.
 
Let
$$ \mathcal U = L_k(\mathfrak{sl}_3). \left( R_{top}  \otimes   \Pi_{-1, -1} (\lambda_1, \lambda_2)_{top}\right).$$ 
Therefore for the proof of irreducibility, it suffices to prove the following.
$$ (*) \quad   R \otimes   \Pi_{-1, -1} (\lambda_1, \lambda_2)  = \mathcal U$$

Define the following vectors in $L_k(\mathfrak{sl}_3)$:
\bea
R_J &=& h_1 \nonumber \\
R_{G^+} &=& f_1  \nonumber \\
R_{G^-} & =& e_3 (-1) f_2 \nonumber \\
R_{L} & = & L_{sug}= L +\frac{1}{2}\partial J +\frac{k}{3}\partial c_2 -\frac{1}{2} \partial d_1 -\frac{1}{2} \partial d_2 + \frac{1}{2}:(c_1 d_1 + c_2 d_2): \nonumber  
\eea

We have the following spanning set for $R$:

 \bea \nonumber && L (-k_1) \cdots L( -k_i) G^-  (-l_1) \cdots G^. (-l_j)  G^+ (-n_1) \cdots   G^+(-n_r)      J(-m_1) \cdots   J (-m_s)  v  \nonumber 
  \eea
where  $v \in L_{top}$  and 
$$ k_1 \ge \cdots \ge k_i \ge 2, l_1 \ge \cdots \ge l_j \ge 2,  n_1 \ge \cdots \ge n_r \ge 1, m_1 \ge \cdots \ge m_s \ge 1. $$

Using induction, Theorem \ref{almost-sl3} and  Lemma  \ref{ired-relaxed-1}  we get:
  \bea && L_{sug} (-k_1) \cdots L_{sug} ( -k_i) R_{G^- } (-l_1) \cdots R_{G^.} (-l_j)  \nonumber \\ &&  R_{G^+} (-n_1) \cdots   R_{G^+}(-n_r)      R_{J}(-m_1) \cdots   R_{J} (-m_s)  ( v  \otimes e^{- \frac{1}{2} d_1 - \frac{1}{2} d_1 + t_1 c_1 + t_2  c_2 }   )    \nonumber  \\
  &=&  \nu L (-k_1) \cdots L( -k_i) G^-  (-l_1) \cdots G^. (-l_j)  G^+ (-n_1) \cdots   G^+(-n_r)  \nonumber  \\   && \cdot  J(-m_1) \cdots   J (-m_s)  v  \otimes e^{- \frac{1}{2} d_1 - \frac{1}{2} d_1 + t_1 c_1 + t_2  c_2 }     + w_1 \nonumber 
   \eea
 where $\nu \ne 0$, $w_1 \in \mathcal U$.
 
 This implies that 
$$ w \otimes e^{- \frac{1}{2} d_1 - \frac{1}{2} d_1 + t_1 c_1 + t_2  c_2 }  \in \mathcal U $$
 for every $w \in L$, $t_1, t_2 \in {\C}$, $t_i  -\lambda_i \in {\Z}$, $i=1,2$.
 
  Now we apply  again  Lemma      \ref{ired-relaxed-1}  and get  $(*)$.
This proves the theorem in the case $k\ne -3$.

For the proof of the critical level case $k=-3$ is similar. The difference  is that then the
 spanning set for $L$ is "smaller"
 \bea \nonumber &&   G^-  (-l_1) \cdots G^. (-l_j)  G^+ (-n_1) \cdots   G^+(-n_r)      J(-m_1) \cdots   J (-m_s)  v  \nonumber 
  \eea
where  $v \in L_{top}$  and 
$$   l_1 \ge \cdots \ge l_j \ge 2,  n_1 \ge \cdots \ge n_r \ge 1, m_1 \ge \cdots \ge m_s \ge 1. $$

\end{proof}

\begin{corollary}
Assume that $U \subset  R_{top} \otimes \Pi_{-1,-1} (\lambda_1, \lambda_2) _{top}$ is irreducible $\mathfrak{sl}_3$--module. Then
$\mathcal L(U):=V^k(\mathfrak{sl}_3). U$ is an irreducible $V^k(\mathfrak{sl}_3)$--module.
\end{corollary}
\section{The structure of the $\mathfrak{sl}_3$--module  $L_{top} \otimes \Pi_{-1,-1} (\lambda_1, \lambda_2) _{top}$}\label{top}

\subsection{ The relaxed $\mathcal W^k$--modules $R_M(\lambda)$ \cite{AKR-2020}.}

  Let $Z^k$ denotes  the  simple Zamolodchikov algebra  with generators  $T, W$ and central charge $c^{Z}_k = -\frac{2 (3k+5) (4k+9)}{k+3}$. The  relaxed  $\mathcal W^k$--modules $R_M(\lambda)$ is constructed in \cite{AKR-2020}
  as 
$  R_M(\lambda) = M \otimes \Pi_{-1} (\lambda)$
where $M$ is an irreducible  highest weight $Z^k$--module   with  highest weight vector $u$ of $T_0$ eigenvalue $\Delta$ and $W_0$ eigenvalue $w$. Define the polynomial 
$$p^{w, \Delta} (x)  =w  - (k+2) (k+3) \Delta + [(k+3) \Delta -2 (k+2)^2] x  +3 (k+2) x^2 - x^3 $$
 (see  the formula (5.17) in \cite{AKR-2020}). 
Then $R_M(\lambda)_{top} \cong  \mbox{span}_{\C} \{  u \otimes e^{-j + \overline{n} c}  \mid n \in {\Z} \}$ and the action of Zhu's algebra $A(\mathcal W_k)$ is uniquely determined by
\bea  G^+ (0)  ( u \otimes e^{-j +  {\overline n}  c}   ) &=& u \otimes e^{-j + \overline{ n +1}  c}, \nonumber \\
 G^- (0)  ( u \otimes e^{-j + {\overline n}  c}   ) &=&  p^{w, \Delta} (\overline n ) u \otimes e^{-j +  \overline{n-1}  c},  \nonumber \\
 J(0)   ( u \otimes e^{-j + (\lambda + n) c}   ) &=& (\overline  n - \frac{2k+3}{3})  ( u \otimes e^{-j +  \overline n c}   )  
 \nonumber \\
 \widetilde L(0)  ( u \otimes e^{-j + \overline c}   ) &=& (\Delta + \frac{2k+3}{3})  ( u \otimes e^{-j +  \overline n c}   ). 
\nonumber 
  \eea
  where $\overline n = n + \lambda$.
It is proved in \cite[Theorem 5.12]{AKR-2020} that $ R_M(\lambda)$ is irreducible  $\mathcal W^k$--module if and only if   
$p^{w, \Delta} (\overline n) \ne 0$ 
 for any $n \in {\Z}$.
  When $2 k +3 \notin {\Z}_{\ge 0}$ and $M$ is  $Z_k$--module, then $R_M(\lambda)$ is a module for the  simple vertex algebra  $\mathcal W_k$.

\subsection{  The $\mathfrak{sl}_3$--action on $R_M(\lambda)_{top} \otimes  \Pi_{-1,-1} (\lambda_1, \lambda_2) _{top}$}

The basis of  $R_M(\lambda)_{top} \otimes  \Pi_{-1,-1} (\lambda_1, \lambda_2) _{top}$ is
$$ v[n,m,p] := u \otimes e^{-j + \overline n  c} \otimes e ^{-\tfrac{d_1} {2} - \tfrac{d_2} {2}  + \overline m  c_1 +  \overline p  c_2}, \quad (m,n, p \in {\Z}) $$
where $\overline  m  = m + \lambda_1$, $ \overline  n  = n + \lambda$, $\overline  p = p + \lambda_2$. 
  Identify $ x = x(0)$ for $x \in \mathfrak{sl}_3$. We have
\bea
e_1 v[n,m,p] &=& ( \overline{m } - \tfrac{1}{2}  )  v[n, m-1, p + 1],  \nonumber \\
e_2 v[n,m,p] &=& v[n, m+1, p],  \nonumber  \\
e_3  v[n,m,p] &=& v[n, m, p+1],   \nonumber  \\
h_1 v[n,m,p] &=&(  - 2\overline  n   - ( \overline{ m}    - \tfrac{1}{2} ) + \overline  p  +  \tfrac{5 (2k+3) }{6}  )   v[n, m, p], \nonumber  \\
h_2 v[n,m,p] &=&( \overline n + 2 \overline m + \overline p - \tfrac{8 k + 9}{6})   v[n, m, p],  \nonumber  \\
f_1 v[n,m,p] &=& v[n+1,m, p] - (2 \overline n - \overline p - \tfrac{5 (2k+3) }{6}) v[n, m+1, p-1],   \nonumber \\
f_2 v[n,m,p] &=&  p^{w, \Delta} (\overline n ) v[n-1, m, p-1] +  (\overline{m} - \tfrac{1}{2} ) ( \tfrac{2 (2k+3) }{3} -  \overline n - \overline m - \overline p) v[n,m-1,p]       \nonumber  \\
f_3  v[n,m,p] &=&  A   v[n, m, p-1] +   B  v[ n+1, m-1, p]  + C v[n-1, m +1, p-2   ]\nonumber  \\
 A &= &  p^{w, \Delta} (\overline n +1) - p^{w, \Delta} (\overline n )  
  - (2 \overline n - \overline p - \tfrac{5 (2k+3) }{6})   ( \tfrac{2 (2k+3) }{3} -  \overline n - \overline m - \overline p) \nonumber \\
 B&  =& - (\overline{m} - \tfrac{1}{2} )   \nonumber \\
 C&=& -p^{w, \Delta} (\overline n ).   \nonumber
\eea
 
Let  $\mathcal U = U(\mathfrak{sl}_3). v[n,m,p]$.  By direct calculation we get
 \bea
 e_1 f_1 v[n,m,p] &=& A_1 v[n,m,p] + (\overline m - \tfrac{1}{2})  v[n+1,m-1,p+1] \quad \nonumber \\
 && A_1 = - (\overline m + \tfrac{1}{2}) (2 \overline n - \overline p -\tfrac{5}{6} (2k +3)),\nonumber \\
 e_2 f_2  v[n,m,p] &=&   A_2 v[n,m,p] +  p^{w, \Delta} (\overline n )  v[n-1,m+1,p-1] \nonumber \\
 && A_2 =  (\overline{m} - \tfrac{1}{2} ) ( \tfrac{2 (2k+3) }{3} -  \overline n - \overline m - \overline p), \nonumber \\
  e_3 f_3  v[n,m,p] &=& A v[n,m,p]  - (\overline{m} - \tfrac{1}{2} )  v[n+1,m-1, p+1] \nonumber \\ &&  -p^{w, \Delta} (\overline n )  v[n-1,m+1, p-1],  \nonumber 
 \eea
 which implies that $v[n-1,m+1,p-1], v[n+1,m-1,p+1] \in \mathcal U$. Using induction we get:
 \bea &&Z_i:=v[n-i, m+i,p-i] \in \mathcal U \quad \forall i \in {\Z}. \label{claim-1} \eea
 Note that all vectors in (\ref{claim-1}) have the  weight   $(r_1 ', r_2 ') =  ( r_1 + \tfrac{1}{2} +  \tfrac{5 (2k+3) }{6},    r_2 - \tfrac{8 k + 9}{6}) $, where
  $r_1 = -2n - m +p$, $r_2 = n + 2m + p$. 
  
   Let $V[r_1,r_2]$ be set of vectors of  $R_M(\lambda)_{top} \otimes  \Pi_{-1,-1} (\lambda_1, \lambda_2) _{top}$
 having weight $(r_1 ', r_2 ')$. We have proved that $V[r_1,r_2] \subset \mathcal U$.

\begin{remark}One sees that: 
\begin{itemize}
\item $\mathcal U$ has all infinite-dimensional weight spaces; 
 
\item $\mathcal U$  is not a Gelfand-Tsetlin module. 
\item So our modules are of the same type as modules in \cite{FLLZ-2016} (i.e.,  infinite-dimensional weight spaces, not Gelfand-Tsetlin module). But we don't see how to construct an isomorphism (or embedding), so it is possible that we discovered  new type of modules with infinite-dimensional weight spaces. This deserves further investigation. \end{itemize}
\end{remark}

\begin{conjecture} \label{conj-1}
Assume that  $$p^{w, \Delta} (\overline n )  \ne 0,  \overline m - \tfrac{1}{2}\ne 0, 
2 \overline n - \overline p - \tfrac{5 (2k+3) }{6} \ne 0,   \tfrac{2 (2k+3) }{3} -  \overline n - \overline m - \overline p\ne 0$$ for all $n, m, p \in {\Z}$. 

Then
 $R_M(\lambda)_{top} \otimes  \Pi_{-1,-1} (\lambda_1, \lambda_2) _{top}$  is an irreducible $\mathfrak{sl}_3$--module.
\end{conjecture}

 \section{ The case $\dim R_{top} < \infty$}
 \label{relaxed-fd}
 Let us consider the case of highest weight $\mathcal W^k$--modules $R$ such that $R_{top}$ is finite-dimensional.   Let $L[x,y]$ be the irreducible highest weight $\mathcal W^k$--module with highest weight $(x,y) \in {\C} ^2$  and highest weight vector $v_{x,y}$ such that for every $n \in {\Z}_{\ge 0}$:
 $$ J(n) v_{x,y} = x \delta_{n,0} v_{x,y}, \  \widetilde L(0) v_{x,y} = y  \delta_{n,0}  v_{x,y}, $$
 $$ G^+ (n+1) v_{x,y} = G^- (n) v_{x,y}=0.$$
 The Zhu algebra $A(\mathcal W^k)$ is generated by $[G^+], [G^-], [J], [\widetilde L]$ and it is isomorphic to a quotient of Smith algebra (cf. \cite{Ar}, \cite{AK}). 
 We have
 $$ G^-(0) G^+ (0) ^i v_{x,y} = h_i(x,y) G^+ (0) ^{i-1} v_{x,y},  $$
 where (cf. \cite{Ar}, \cite{AK})  $$ h_i(x,y) =  -i^2+ki-3xi+3i-3x^2-k+2kx+6x+ky+3y-2.$$
 
 Assume that $L[x,y]_{top} = N  < \infty$, then
 $$L[x,y]_{top} = \mbox{span}_{\C} \{ G^+ (0) ^i v_{x,y} \ \vert \ 0 \le i \le N-1\}, $$
 and $h_N(x,y) =0$.
 
 This implies that $L[x,y]_{top}  \otimes  \Pi_{-1,-1} (\lambda_1, \lambda_2) _{top}$ has the following basis:
 $$w[i, m, p]:=   G^+ (0) ^i v_{x,y}  \otimes e ^{-\tfrac{d_1} {2} - \tfrac{d_2} {2}  + \overline m  c_1 +  \overline p  c_2}, $$
 where $i =0, \dots, N-1$, $ m, p \in {\Z}$.
 The $\mathfrak{sl}_3$ action on $L[x,y]_{top}  \otimes  \Pi_{-1,-1} (\lambda_1, \lambda_2) _{top}$ is given by 
\bea
e_1 w[i,m,p] &=& ( \overline{m } - \tfrac{1}{2}  )  w[i, m-1, p + 1],  \nonumber \\
e_2 w[i,m,p] &=& w[i, m+1, p],  \nonumber  \\
e_3  w[i,m,p] &=& w[i, m, p+1],   \nonumber  \\
h_1 w[i,m,p] &=&(  - 2 (x +i)    - ( \overline{ m}    - \tfrac{1}{2} ) + \overline  p  +  \tfrac{5 (2k+3) }{6}  )   w[i, m, p], \nonumber  \\
h_2 w[n,m,p] &=&( x + i + 2 \overline m + \overline p - \tfrac{8 k + 9}{6})   w[i, m, p],  \nonumber  \\
f_1 w[i ,m,p] &=& v[i+1,m, p] - (2 (x+i)  - \overline p - \tfrac{5 (2k+3) }{6}) w[i, m+1, p-1],  \nonumber \\ &&   \quad (0 \le i \le N-2),  \nonumber \\
f_1 w[N-1 ,m,p] &=&  - (2 (x+N-1)  - \overline p - \tfrac{5 (2k+3) }{6}) w[N-1, m+1, p-1]     \nonumber \\
f_2 w[i,m,p] &=&  h_i(x,y) v[i-1, m, p-1] +  \nonumber \\   & & \quad (\overline{m} - \tfrac{1}{2} ) ( \tfrac{2 (2k+3) }{3} -  (x+i)  - \overline m - \overline p) w[i,m-1,p],       \nonumber  \\
&&    \quad (1 \le i \le N-1)   \nonumber  \\
f_2 w[0,m,p] &=&    (\overline{m} - \tfrac{1}{2} ) ( \tfrac{2 (2k+3) }{3} -  x  - \overline m - \overline p) w[0,m-1,p]       \nonumber  \\
f_3  w[n,m,p] &=&  [f_2, f_1] w[n,m,p] \nonumber  
\eea

  In the special, but very important  case $N=1$, we get a realization of   a family of irreducible, relaxed  modules for arbitrary level $k$:
 \begin{theorem} Assume that  $$h_1(x,y) = (3 + 2 k) x - 3 x^2 + (3 + k) y = 0,$$ and  for all $m, p \in {\Z}$: $$ \overline m -\tfrac{1}{2} \ne 0,   \tfrac{2 (2k+3) }{3} -  x  - \overline m - \overline p \ne 0,  2 x   - \overline p - \tfrac{5 (2k+3) }{6} \ne 0.  $$
  Then we have:
 \begin{itemize}
 \item  $L[x,y]_{top}  \otimes  \Pi_{-1,-1} (\lambda_1, \lambda_2) _{top}$ is an irreducible $\mathfrak{sl}_3$--module.
 \item    $L[x,y] \otimes  \Pi_{-1,-1} (\lambda_1, \lambda_2) $ is an irreducible $V^k(\mathfrak{sl}_3)$--module.
  \item
 If $L[x,y]$ is an $\mathcal W_k$--module and $k \notin {\Z}_{\ge 0}$, then  $L[x,y] \otimes  \Pi_{-1,-1} (\lambda_1, \lambda_2) $ is an irreducible $L_k(\mathfrak{sl}_3)$--module.
 \end{itemize}
 \end{theorem}
 \begin{proof}
 Note that in the case $N=1$, the actions of $f_1$ and $f_2$ are simplified:
 \bea f_1 w[0 ,m,p] &=&  - (2 x   - \overline p - \tfrac{5 (2k+3) }{6}) w[0, m+1, p-1] \nonumber  \\
 f_2 w[0,m,p] &=&    (\overline{m} - \tfrac{1}{2} ) ( \tfrac{2 (2k+3) }{3} -  x  - \overline m - \overline p) w[0,m-1,p] \nonumber \eea
The proof of irreducibility of  $L[x,y]_{top}  \otimes  \Pi_{-1,-1} (\lambda_1, \lambda_2) _{top}$ is now easy. The second and third assertions follow from  Theorem
\ref{thm:irred}.
 \end{proof}

 Next we should look at the case $N\ge 2$. Consider the vector space
 $$W[r_1,r_2] = \mbox{span}_{\C}\{w[i, m-i, p+i], i= 0, \cdots, N-1\} $$
 of vectors of  weight   
 $$(r_1 ', r_2 ') =  ( r_1 -2 x -\lambda_2 + \lambda_1+  \tfrac{1}{2} +  \tfrac{5 (2k+3)   }{6},    r_2  + 2\lambda_1-\lambda_2- \tfrac{8 k + 9}{6}), $$ where
  $r_1 =  - m +p$, $r_2 =   2m + p$. 
  
  Denote by ${\bf C}(\mathfrak h)$ the centralizator of the Cartan subalgebra ${\mathfrak h}$ in $U(\mathfrak{sl}_3)$.
  
  \begin{lemma} \label{cent-irred} Let $N >1$. Assume that 
  $$h_N(x,y)   = 0, \ h_j(x,y) \ne 0 \quad (1 \le j \le N-1),$$ and that   for all $m, p \in {\Z}$: $$ \overline m -\tfrac{1}{2} \ne 0,   \tfrac{2 (2k+3) }{3} -  x  - \overline m - \overline p \ne 0,  2 x   - \overline p - \tfrac{5 (2k+3) }{6} \ne 0.  $$
Then $W[r_1,r_2] $ is an irreducible ${\bf C}(\mathfrak h)$--module.
  \end{lemma}
  \begin{proof}
  Define $Z_i := w[i, m-i, p+i]$, $i=0, \dots, N-1$,  and elements $u_i:= e_i f_i \in {\bf C}(\mathfrak h)$, $i=1,2$. We have:
  \bea
  u_1 Z_{N-1} &=& \underbrace{ -(\overline m -N+1)  (2 x+N-1  - \overline p - \tfrac{5 (2k+3) }{6}) }_{=a_{N-1}}  Z_{N-1} \nonumber  \\
  u_1 Z_{i} &=&\underbrace{ - (\overline m + \tfrac{1}{2})  (2 x+i  - \overline p - \tfrac{5 (2k+3) }{6})}_{=a_i}  Z_i  
  - (\overline m - \tfrac{1}{2})  Z_{i+1} \nonumber  \\
  && \quad (0  \le i \le N-2). \nonumber
  \eea
  Similarly we get
   \bea
  u_2 Z_{0} &=&\underbrace{ (\overline{m} - \tfrac{1}{2} ) ( \tfrac{2 (2k+3) }{3} -  x  - \overline m - \overline p)}_{=b_0} Z_{0} \nonumber  \\
  u_2 Z_{i} &=&   \underbrace{(\overline{m} -i - \tfrac{1}{2} ) ( \tfrac{2 (2k+3) }{3} -  (x+i)  - \overline m - \overline p)}_{=b_i} Z_i  
 + h_{i} (x,y) Z_{i-1} \nonumber  \\
  && \quad (1 \le i \le N-1). \nonumber
  \eea
  \begin{itemize}
  \item[(1)] First we see that each $Z_i$ is cyclic.
  \item[] Let $U =  {\bf C}(\mathfrak h). Z_i$.   If $ i < N-1$, we take any $N-1 \ge j >i$.  Then
   $$ (u_1- b_{j-1}) \cdots (u_1 - b_i) Z_i  = \nu Z_j \quad (\nu \ne 0).$$ 
   This implies that  $Z_i, \dots, Z_{N-1} \in U$.
   If $i >0$, we take any $ 0 \le  j  < i$ and get
   $$  (u_1-b_0) \cdots (u_1 - b_{i-1}) Z_i = \nu Z_{j} \quad (\nu \ne 0).$$ 
   This way we get $U = W[r_1,r_2] $.
   
     \item[(2)] Now we prove that $W[r_1,r_2] $ is irreducible.
     \item[] Let $W \subset W[r_1,r_2] $ by any non-zero  submodule. 
  Let 
  $$ w_1=  c_1 Z_{i_1} + \cdots + c_{\ell} Z_{i_{\ell}} \in W$$
  where  all $c_i \ne 0$ and $0 \le i_1 <   \cdots  < i_{\ell} \le N. $
  By applying  $u_1- a_{i_1}$, we get
 $$ w_2 = (u_1 - a_{i_1}) w_1 =   c'_1 Z_{j_1} + \cdots + c'_{\ell'} Z_{j_{\ell'}} \in W$$
 where $ c'_1 \ne 0, j_1 > i_1$.
 
 By continuing this approach, after finitely many steps, we get that $Z_i \in W$ for certain $0 \le i \le N-1$, implying by (1)  that $W = W[r_1,r_2] $.
  
      This proves the irreducibility.
  \end{itemize}
  \end{proof}
\begin{theorem}  Assume that the conditions of Lemma \ref{cent-irred} are satisfied. Then $L[x,y]_{top}  \otimes  \Pi_{-1,-1} (\lambda_1, \lambda_2) _{top}$ is an irreducible $\mathfrak{sl}_3$--module and $L[x,y]  \otimes  \Pi_{-1,-1} (\lambda_1, \lambda_2) $ is an irreducible $V^k(\mathfrak{sl}_3)$--module.

If $L[x,y]$ is an $\mathcal W_k$--module and $k \notin {\Z}_{\ge 0}$, then  $L[x,y] \otimes  \Pi_{-1,-1} (\lambda_1, \lambda_2) $ is an irreducible $L_k(\mathfrak{sl}_3)$--module.
\end{theorem} 
  \begin{proof} One  sees 
 $L[x,y]_{top}  \otimes  \Pi_{-1,-1} (\lambda_1, \lambda_2) _{top}$ is a direct sum of $N$--dimensional weight spaces $W[r_1,r_2]$ as above, and all operators 
 $e_i(0), f_i(0)$, $i=1,2$  act injectively  on $W[r_1,r_2]$. Then applying Lemma  \ref{cent-irred} we get that $L[x,y]_{top}  \otimes  \Pi_{-1,-1} (\lambda_1, \lambda_2) _{top}$ is irreducible $\mathfrak{sl}_3$--module. The irreducibility of $L[x,y]  \otimes  \Pi_{-1,-1} (\lambda_1, \lambda_2) $ follows by   Theorem
 \ref{thm:irred}.
  \end{proof}

 \section{Screening operators  and logarithmic modules for $\SBP$}
 \label{log-bp}


Let $T(z), W(z)$ be generators of $\mathcal{Z}^k = \mathcal{W}^k(\mathfrak{sl}_3, f_\mathrm{prin})$ with conformal weight $2, 3$ such that $W$ is primary to $T$, and $T, W$ satisfy the same OPE relations as in \cite{AKR-2020}. 
We set
\begin{align*}
T(z) = \sum_{n \in \Z}T_n z^{-n-2},\quad
W(z) = \sum_{n \in \Z}W_n z^{-n-3}.
\end{align*}
\begin{proposition} \label{construction-bp-1-scr}
There exist a highest weight $\mathcal{Z}^k$--module $\widetilde L$ with the highest weight vector $v_0$ such that
\begin{align*}
T_0 v_0 = \frac{4k+9}{3}v_0,\quad
W_0 v_0 = -\frac{(k+3)(4k+9)(5k+12)}{27}v_0,
\end{align*}
and that 
\begin{align*}
V = \int (v_0 \otimes \mathrm{e}^{-\frac{2k+3}{6}c+\frac{1}{2}d})(z) dz : \mathcal{Z}^k \otimes \Pi(0) \rightarrow \widetilde L \otimes \Pi(0). \mathrm{e}^{-\frac{2k+3}{6}c+\frac{1}{2}d}
\end{align*}
is a screening operator of $\BP \hookrightarrow \mathcal{Z}^k \otimes \Pi(0)$.
\begin{proof}
We have
 \begin{align*}
\BP
&=\ \Ker \int \mathrm{e}^{-\alpha_1/(k+3)}(z) dz \cap  \Ker \int \beta(z)\mathrm{e}^{-\alpha_2/(k+3)}(z) dz\\
&=\ \Ker \int \mathrm{e}^{-\widetilde{\alpha}_1/(k+3)}(z) dz \cap  \Ker \int \mathrm{e}^{-\widetilde{\alpha}_2/(k+3)}(z) dz \cap \Ker \int \mathrm{e}^x(z) dz,
 \end{align*}
 where we use $\beta \mapsto \mathrm{e}^{x+y}$, $\gamma \mapsto -\NO{x\mathrm{e}^{-x-y}}$ and
 \begin{align*}
 \widetilde{\alpha}_1 = \alpha_1,\quad
 \widetilde{\alpha}_2 = \alpha_2 -(k+3)(x+y).
 \end{align*}
Let $\mathcal{Z}^k = \mathcal{W}^k(\mathfrak{sl}_3, f_\mathrm{prin})$. Since
\begin{align*}
\mathcal{Z}^k = \Ker \int \mathrm{e}^{-\widetilde{\alpha}_1/(k+3)}(z) dz \cap  \Ker \int \mathrm{e}^{-\widetilde{\alpha}_2/(k+3)}(z) dz,
\end{align*}
we get
\begin{align}\label{eq:BP-to-Z}
\BP \hookrightarrow \mathcal{Z}^k \otimes \Pi(0)
\end{align}
by forgetting the screening operator $\int \mathrm{e}^x(z) dz$, where $\Pi(0)$ is generated by $c, d, \mathrm{e}^c, \mathrm{e}^{-c}$ and
\begin{align*}
c = x+y,\quad
d = \frac{2}{3}\alpha_1 + \frac{4}{3}\alpha_2 -\frac{2k+3}{3}x - \frac{2k+9}{3}y.
\end{align*}
Thus the embedding \eqref{eq:BP-to-Z} commutes with $\int \mathrm{e}^x(z) dz$ and we have
\begin{align*}
x = -\frac{\widetilde{\alpha}_1 + 2\widetilde{\alpha}_2}{3} -\frac{2k+3}{6}c+\frac{1}{2}d
\end{align*}
so that
\begin{align*}
\mathrm{e}^x = \mathrm{e}^p \otimes \mathrm{e}^{-\frac{2k+3}{6}c+\frac{1}{2}d},\quad
p = -\frac{\widetilde{\alpha}_1 + 2\widetilde{\alpha}_2}{3}.
\end{align*}
Since generators $T(z), W(z)$ of $\mathcal{Z}^k$ in \cite{AKR-2020} can be written by
\begin{align*}
&T = \frac{1}{k+3}\left( (k+2)\partial(\widetilde{\alpha}_1+\widetilde{\alpha}_2) + \frac{1}{3}:(\widetilde{\alpha}_1^2 +\widetilde{\alpha}_1\widetilde{\alpha}_2 + \widetilde{\alpha}_2^2): \right),\\
&W = -\frac{(k+2)^2}{6}\partial^2(\widetilde{\alpha}_1-\widetilde{\alpha}_2) - \frac{k+2}{6}:((\partial\widetilde{\alpha}_1)(2\widetilde{\alpha}_1+\widetilde{\alpha}_2) - (\widetilde{\alpha}_1+2\widetilde{\alpha}_2)\partial\widetilde{\alpha}_2):\\
&\qquad - \frac{1}{27}:(2\widetilde{\alpha}_1^3 + 3\widetilde{\alpha}_1^2\widetilde{\alpha}_2 - 3\widetilde{\alpha}_1\widetilde{\alpha}_2^2 - 2\widetilde{\alpha}_2^3):,
\end{align*}
it follows that
\begin{align*}
&T(z)\mathrm{e}^p(w) \sim \frac{(4k+9)\mathrm{e}^p(w)}{3(z-w)^2} - \frac{\NO{(\widetilde{\alpha}_1+2\widetilde{\alpha}_2)\mathrm{e}^p}(w)}{3(z-w)},\\
&W(z)\mathrm{e}^p(w) \sim - \frac{(k+3)(4k+9)(5k+12)\mathrm{e}^p(w)}{27(z-w)^3}\\
&- \frac{(k+3)(5k+12)\NO{(\widetilde{\alpha}_1+2\widetilde{\alpha}_2)\mathrm{e}^p}(w)}{18(z-w)^2}\\
&+ \frac{(k+3)\NO{(\widetilde{\alpha}^2-2\widetilde{\alpha}_1\widetilde{\alpha}_1\widetilde{\alpha}_2-2\widetilde{\alpha}_2^2+(4k+9)\partial\widetilde{\alpha}_1+2(k+3)\partial\widetilde{\alpha}_2)\mathrm{e}^p}(w)}{9(z-w)}.
\end{align*}
Therefore
\begin{align*}
T_0 \mathrm{e}^p = \frac{4k+9}{3}\mathrm{e}^p,\quad
W_0 \mathrm{e}^p = -\frac{(k+3)(4k+9)(5k+12)}{27}\mathrm{e}^p.
\end{align*}
The assertion now follows by taking $v_0 = e^{p}$ and  $\widetilde L = \mathcal Z^k . e^{p} $.
\end{proof}
\end{proposition} 
 
Let $i = -\frac{2k+3}{6} c + \tfrac{d}{2}$. Then $e^x = e^p \otimes e^{i}$. Then $\widetilde L = \mathcal Z^k . e^{p} $ is an highest weight $\mathcal Z^k$--module with highest weight $( \tfrac{4k+9}{3}, -\tfrac{ (4 k+9) (k+3) (5 k + 12)}{27})$ from Proposition
 \ref{construction-bp-1-scr}.
 We denote its irreducible quotient by $L_{1 1; 1 2}$. By   Proposition
 \ref{construction-bp-1-scr}, we have that:
 \begin{itemize}
 \item Let $\widetilde{ \mathcal Y} (\cdot, z)$ be an  intertwining operator of the type
 $$ { \widetilde L \otimes \Pi(0) e^{i} \choose  \widetilde L \otimes \Pi(0) e^{i}  \ \ \mathcal Z^k \otimes \Pi(0) }.$$
Then $e^x(z) = \widetilde{ \mathcal Y} (e^x, z)$.
 
 \item  $e^x _0$  is a $\BP$--homomorphism and $e^x _0 \vert \BP = 0$.
 \end{itemize}
Next we notice that  
 $$   I   { \widetilde L \otimes \Pi(0) e^{i} \choose  \widetilde L \otimes \Pi(0) e^{i}  \ \ \mathcal Z^k \otimes \Pi(0) } \cong   I   { \widetilde L \otimes \Pi(0) e^{i} \choose       \mathcal Z^k \otimes \Pi(0) \ \ \widetilde L \otimes \Pi(0) e^{i} },   $$
 and
 $$ e^{x}(z)  v = e^{z D} Y_{ \widetilde L \otimes \Pi(0) e^{i}}  (v,-z)  (e^p \otimes e^{i})$$
 where
 $Y_{ \widetilde L \otimes \Pi(0) e^{i}} (\cdot, z)$ is the vertex operator which defines the module structure on $\widetilde L \otimes \Pi(0) e^{i}$. For $v \in \BP \hookrightarrow \mathcal Z^k \otimes \Pi(0)$, we get 
\bea
  \int \left( e^{x}(z)  v \right) dz =  \int \left( e^{-z D} Y_{ \widetilde L \otimes \Pi(0) e^{i}}  (v,-z)  (e^p \otimes e^{i}) \right) dz = 0 \label{rel-01}
\eea

 Assume that $L_{1 1; 1 2}$ is an (irreducible) $\mathcal Z_k$--module. Then we have 
 
 $$ \dim  I   {  L_{1 1; 1 2} \otimes \Pi(0) e^{i} \choose  L_{1 1; 1 2} \otimes \Pi(0) e^{i}  \ \ \mathcal Z_k \otimes \Pi(0) }= \dim   I   { L_{1 1; 1 2}  \otimes \Pi(0) e^{i} \choose       \mathcal Z_k \otimes \Pi(0) \ \ L_{1 1; 1 2}  \otimes \Pi(0) e^{i} } = 1. $$
 This implies that there is a unique, up to a scalar factor, intertwining operator $\mathcal Y(\cdot, z)$ of type ${  L_{1 1; 1 2} \otimes \Pi(0) e^{i} \choose  L_{1 1; 1 2} \otimes \Pi(0) e^{i}  \ \ \mathcal Z_k \otimes \Pi(0) }$ which is obtained as a transpose of the vertex operator $Y_{ L_{1 1; 1 2} \otimes \Pi(0) e^{i}} (\cdot, z)$  defining  the unique structure of a module on $L_{1 1; 1 2} \otimes \Pi(0) e^{i}$.
 
 As a consequence, we get:
 
 $$  \mathcal Y(v[1 1;1 2] \otimes e^i, z) v =   e^{z D} Y_{ L_{1 1; 1 2} \otimes \Pi(0) e^{i}}  (v,-z)  (v[1 1;1 2] \otimes e^{i})  $$
 where
 $Y_{  L_{1 1; 1 2} \otimes \Pi(0) e^{i}} (\cdot, z)$ is the vertex operator which defines the  $\mathcal Z_k \otimes \Pi(0)$ module structure on $L_{1 1; 1 2} \otimes \Pi(0) e^{i}$.
 
 Since $L_{1 1; 1 2} \otimes \Pi(0) e^{i}$ is a simple quotient of  $\widetilde L \otimes \Pi(0) e^{i}$, relation (\ref{rel-01}) implies that for $v \in \SBP$: 
 
 \bea
 && \int \left( \mathcal Y(v[1 1;1 2] \otimes e^i, z) v \right)  dz \nonumber \\ &=& \int \left( e^{z D} Y_{ L_{1 1; 1 2} \otimes \Pi(0) e^{i}}  (v,-z)  (v[1 1;1 2] \otimes e^{i}) 
\right) dz = 0   \label{rel-02}
 \eea
 In this way we have proved:

 \begin{proposition}  Assume that $L_{1 1; 1 2}$ is a $\mathcal Z_k$--module.
   Let $v[1 1;1 2]$ denotes the highest weight vector of $L_{1 1; 1 2}$.
 Let $$ Q= \int \mathcal Y(v[1 1;1 2] \otimes e^i, z) dz :  \mathcal Z_k \otimes \Pi(0) \rightarrow L_{1 1; 1 2} \otimes \Pi(0). e^i . $$
 Then $Q$ is an non-trivial screening operator which commutes with the action of $\SBP$ and we have
 $$ Q e^{-c} = v[1 1; 1 2] \otimes e^{i-c}. $$
  \end{proposition}
     
   For the construction of logarithmic modules we shall use methods  from   \cite{A-2019} and  \cite{AdM-2009}.
   
   Consider now extended vertex algebra
   $$ \mathcal V_k = \mathcal Z_k  \otimes \Pi(0) \oplus L_{1 1; 1 2} \otimes \Pi(0) e^{i}. $$
 Let $\mathcal M$ be any $\mathcal V_k$--module. Define
 
 $$(\widetilde{\mathcal M }, \widetilde  Y_{\mathcal M}(\cdot , z) ) :=( {\mathcal M}  , Y_{\mathcal M}(\Delta(s, z)\cdot , z) ),  $$
 where
 $$ s = v[1 1; 1 2] \otimes e^{i},  $$
 $$  \Delta(s, z) =  z^{s_0} \exp \left(\sum_{n=1} ^{\infty} \frac{s(n)}{-n} (-z) ^{-n} \right). $$

 Since
 $$ \SBP \subset \mbox{Ker} (Q  : \mathcal V_k \rightarrow \mathcal V_k),$$
we conclude that  $\widetilde  Y_{\mathcal M}(\cdot, z)$ defines on $ \widetilde{\mathcal M} $ the structure of  a $\SBP$--module.
The action of $\widetilde L^{BP} (0) $ on $ \widetilde{\mathcal M }$ is
$$ \widetilde L^{BP} (0) = L^{BP} (0)  + Q .$$
 
\begin{proposition}
Assume that $Z_{1 1; 1 2}$ is a $\mathcal Z_k$--module and $2k +3 \notin {\Z}_{\ge 0}$.
\begin{itemize}
\item  Assume that $\mathcal M$ is a $L^{BP}(0)$ semi--simple $\mathcal V$--module such that  $Q$ acts non-trivially on  $\mathcal M$. Then  $ \widetilde{\mathcal M} $ is a logarithmic $\SBP$--module.

\item $\widetilde {\mathcal V_k}$ is a logarithmic $\SBP$--module of   $ \widetilde L^{BP} (0)$ nilpotent rank two.
\end{itemize}

\end{proposition}
  
\begin{remark}
Let $k = -3 + \frac{u}{v}$ be an admissible level, that is $u, v$ are coprime positive integers and $u\geq 3$. Then $Z_{1 1; 1 2}$ is a $\mathcal Z_k$--module if $v>3$. 
\end{remark}  
  
 \section{Logarithmic $L_k(\mathfrak{sl}_3)$--modules}\label{log-sl}
 
 First we recall that the $\beta \gamma$ vertex operator algebra $\mathcal S(1)$ has the logarithmic (projective) --modules $\mathcal P_s$, $s \in {\Z}$,  of $L^{\mathcal S(1)}(0)$--nilpotent rank two (cf. \cite{ACKR}, \cite{AW}).
 We get:

 \begin{proposition}  Assume  $k \notin {\Z}_{\ge 0}$ and  assume that $N$ is irreducible, highest weight or relaxed  $\mathcal W_k$--module.  Then for each $r,s\in {\Z}$, $\lambda \in {\C}$:
 $$N \otimes \mathcal P_s \otimes \Pi(0).e^{\frac{r}{2} d_2 + \lambda c_2}$$
 is an logarithmic $L_k(\mathfrak{sl}_3)$--module of $L_{sug}(0)$--nilpotent rank two.
 \end{proposition}
 \begin{proof}
Since  $L_k(\mathfrak{sl}_3) \hookrightarrow \mathcal W_k \otimes \mathcal S(1) \otimes \Pi(0)$, we get that  $N \otimes \mathcal P_s \otimes \Pi(0).e^{\frac{r}{2} d_2 + \lambda c_2}$ is a $L_k(\mathfrak{sl}_3)$--module.
 
Using 
 (\ref{eq:Sugawara-new}) we get that 
$$ L_{sug} \mapsto L^\mathrm{BP} +\frac{1}{2}\partial J + L^{\mathcal S(1)} +  L^{\Pi(0)}$$
 where $L^\mathrm{BP}, L^{\mathcal S(1)},  L^{\Pi(0)}$ are commuting Virasoro vectors on $\mathcal W_k$, $\mathcal S(1)$, $\Pi(0)$, respectively. Since $J(0), L^\mathrm{BP}(0), L^{\Pi(0)}$ act semisimply,  and $L^{\mathcal S(1)}(0)$ has nilpotent rank two, we conclude that $ L_{sug} (0)$ has nilpotent rank two. The proof follows.
 \end{proof}
      
We can now consider the logarithmic $\SBP$--modules from Section \ref{log-bp}. For simplicity we shall consider only logarithmic module $\widetilde {\mathcal V_k}$.

 \begin{proposition}  \label{prop-admiss-cond}
  Assume that $k = -3 + \frac{u}{v}$ be an admissible level such that is $u, v$ are coprime positive integers and $v\geq 4$.
 
 Then for each $r, s\in {\Z}$, $\lambda \in {\C}$:
 $$\widetilde {\mathcal V_k}  \otimes \mathcal P_s \otimes \Pi(0).e^{\frac{r}{2} d_2 + \lambda c_2}$$
 is a logarithmic $L_k(\mathfrak{sl}_3)$--module of $L_{sug}(0)$--nilpotent rank three.
 \end{proposition}

 \begin{remark}  Let us summarize our result on logarithmic modules.
 \begin{itemize}
 \item $k\in {\Z}_{\ge 0}$, $L_k(\mathfrak{sl}_3)$ is rational and there are no logarithmic modules.
 \item $k \in \tfrac{3}{2} +  {\Z}_{\ge 0}$, $L_k(\mathfrak{sl}_3)$ is admissible, and there exist logarithmic modules of rank two constructed using $\beta \gamma$ projective modules $\mathcal P_s$.  It is expected that there exist logarithmic modules of higher rank. 
 \item If $k$ is admissible  satisfying  condition of Proposition \ref{prop-admiss-cond}. Then
  there exist logarithmic modules of rank three. They are constructed using $\SBP$--modules of type $\widetilde{\mathcal V_k}$ and $\beta \gamma$ modules $\mathcal P_s$. 
  \item If $k$ is generic, there exist logarithmic modules of order three. 
 \end{itemize}
 We hope to study these logarithmic modules in more details in our forthcoming publications, in particular we will investigate the construction of logarithmic modules of possibly higher rank. The eventual goal is to find logarithmic modules that are projective. 
 
 \end{remark}

\subsection*{Acknowledgments}

This work was started  in part during the visit of D.A. to the University of Alberta in  November 2019.

 D.A.   is   partially
supported   by the QuantiXLie Centre of Excellence, a project coffinanced
by the Croatian Government and European Union
through the European Regional Development Fund - the
Competitiveness and Cohesion Operational Programme
(KK.01.1.1.01.0004).

T. C is supported by NSERC $\#$RES0020460.

N. G is supported by JSPS Overseas Research Fellowships.


\vskip10pt {\footnotesize{}{ }\textbf{\footnotesize{}D.A.}{\footnotesize{}:
Department of Mathematics, Faculty of Science, University of Zagreb, Bijeni\v{c}ka 30,
10 000 Zagreb, Croatia; }\texttt{\footnotesize{}adamovic@math.hr}{\footnotesize \par}

\vskip10pt {\footnotesize{}{ }\textbf{\footnotesize{}T.C.}{\footnotesize{}:
Department of Mathematical and Statistical Sciences, University of Alberta, 632 CAB, Edmonton, Alberta, Canada T6G 2G1; }\texttt{\footnotesize{}creutzig@ualberta.ca}{\footnotesize \par}

\vskip10pt {\footnotesize{}{ }\textbf{\footnotesize{}N.G.}{\footnotesize{}:
Department of Mathematical and Statistical Sciences, University of Alberta, 632 CAB, Edmonton, Alberta, Canada T6G 2G1; }\texttt{\footnotesize{}genra@ualberta.ca}{\footnotesize \par}

\end{document}